\documentclass[10pt]{amsart}
\usepackage[matrix,arrow]{xy}
\usepackage{amssymb}
\usepackage{amsmath}
\usepackage{amsfonts}
\usepackage{epsfig}
\usepackage{color}
\usepackage{epstopdf}
\usepackage{graphicx}

\usepackage{amsthm}
\usepackage{enumerate}
\usepackage[mathscr]{eucal}
\usepackage{verbatim}
\usepackage{bm}
\usepackage{tikz-cd}

\usepackage[bookmarks=true,hyperindex,pdftex,colorlinks,citecolor=blue,
linkcolor=blue,urlcolor=blue]{hyperref}
\usepackage{tikz}
\usetikzlibrary{matrix,arrows.meta}

\theoremstyle{plain}
\newtheorem{theorem}{Theorem}[section]
\newtheorem{cor}[theorem]{Corollary}
\newtheorem{prop}[theorem]{Proposition}
\newtheorem{lemma}[theorem]{Lemma}

\theoremstyle{definition}

\newtheorem{example}[theorem]{Example}

\newtheorem{question}[theorem]{Question}
\newtheorem{rem}[theorem]{Remark}
\newtheorem{remark}[theorem]{Remark}
\newtheorem{definition}[theorem]{Definition}

\newcommand{\R}{\mathbb{R}}

\DeclareMathOperator{\re}{Re}

\renewcommand{\subset}{\subseteq}

\title[M-ideals of compact operators and Norm attaining operators]
{M-ideals of compact operators and Norm attaining operators}

\author[M.~Han]{Manwook Han}
\address[M.~Han]{Department of Mathematics, Chungbuk National University, Cheongju, Chungbuk 28644, Republic of Korea}
\email{\texttt{mwhan0828@gmail.com}}

\author[S.~K.~Kim]{Sun Kwang Kim}
\address[S.~K.~Kim]{Department of Mathematics, Chungbuk National University, Cheongju, Chungbuk 28644, Republic of Korea\newline
	\href{http://orcid.org/0000-0002-9402-2002}{ORCID: \texttt{0000-0002-9402-2002}  }}
\email{\texttt{skk@chungbuk.ac.kr}}

\thanks{}


\begin{document}

\begin{abstract} We investigate M-ideals of compact operators and two distinct properties in norm-attaining operator theory related with M-ideals of compact operators called the weak maximizing property and the compact perturbation property. For Banach spaces $X$ and $Y$, it is previously known that if $\mathcal{K}(X,Y)$ is an M-ideal or $(X,Y)$ has the weak maximizing property, then $(X,Y)$ has the adjoint compact perturbation property. We see that their converses are not true, and the condition that $\mathcal{K}(X,Y)$ is an M-ideal does not imply the weak maximizing property, nor vice versa. Nevertheless, we see that all of these are closely related to property $(M)$, and as a consequence, we show that if $\mathcal{K}(\ell_p,Y)$ $(1<p<\infty)$ is an M-ideal, then $(\ell_p,Y)$ has the weak maximizing property. We also prove that $(\ell_1,\ell_1)$ does not have the adjoint compact perturbation property, and neither does $(\ell_1,Y)$ for an infinite dimensional Banach space  $Y$ without an isomorphic copy of $\ell_1$ if  $Y$ does not have the local diameter 2 property. As a consequence, we show that if $Y$ is an infinite dimensional Banach space such that $\mathcal{L}(\ell_1,Y)$ is an M-ideal, then it has the local diameter 2 property. Furthermore, we also studied various geometric properties of Banach spaces such as the Opial property with moduli of asymptotic uniform smoothness and uniform convexity.
\end{abstract}

\maketitle

\section{Introduction}\label{section1}
In their seminal work \cite{AE} of 1972, E. M. Alfsen and E. G. Effros considered distinguished subspaces of a Banach space called M-ideals to study the space by means of geometric and analytic properties of the unit ball of its dual space. Afterwards, a lot of papers regarding M-ideals have been published. Especially there had been many efforts to characterize those Banach space $X$ such that the subspace of compact operators $\mathcal{K}(X)$ is an M-ideal in the space of bounded linear operators $\mathcal{L}(X)$ (see \cite{CJ,HL,H1,H2,L1,L2,SW}), and many important concepts, like property $(M)$, had been introduced. Our goal here is to use these concepts to investigate the weak maximizing property and the compact perturbation property in norm-attaining operator theory.

Let us give necessary background materials to make the paper entirely accessible.  Let $X$ and $Y$ be Banach spaces over the  field $\mathbb{K} = \mathbb{R}$ or $\mathbb{C}$. We use common notations $S_X$, $B_X$ and $X^*$ for the unit sphere, the closed unit ball and the dual space of $X$ respectively. We denote by $\mathcal{L}(X,Y)$ the Banach space of all bounded linear operators from $X$ into $Y$, and by $\mathcal{K}(X,Y)$ the subspace of compact operators. If $X=Y$, then we write simply $\mathcal{K}(X)$ and $\mathcal{L}(X)$ instead of $\mathcal{K}(X,Y)$ and $\mathcal{L}(X,Y)$. Specially, we write  $id_X\in \mathcal{L}(X)$ for the identity operator. We say that an operator $T\in \mathcal{L}(X,Y)$ \textbf{attains its norm} if there exists $x\in S_X$ such that $\|Tx\|=\|T\|=\sup_{z\in B_X} \|Tz\|$.

A closed subspace $J$ of $X$ is said to be an \textbf{L-summand} in $X$ if there is a projection $P$ on $X$ such that $PX=J$ and $\|x\|=\|Px\|+\|(id_X-P)x\|$ for every $x\in X$. $J$ is said to be an \textbf{M-ideal} in $X$ if the annihilator $J^{\perp}$ of $J$ is an L-summand in $X^*$. E. M. Alfsen and E. G. Effros introduced this notion in \cite{AE}, and they found that M-ideals in the self-adjoint part of a $C^*$-algebra are the self adjoint parts of the closed two-sided ideals. Afterwards, it is proved that M-ideals in a $C^*$-algebra are exactly the closed two-sided ideals (see \cite{SW}), and $\mathcal{K}(\ell_p)$ is an M-ideal in $\mathcal{L}(\ell_p)$ if and only if $1<p<\infty$ (see \cite{H1,SW}). The later result is strengthened by C. Cho and W. B. Johnson in \cite{CJ} as that whenever a subspace $X$ of $\ell_p$ ($1<p<\infty$) has the compact approximation property $\mathcal{K}(X)$ is an M-ideal in $\mathcal{L}(X)$. It is worth to note that $\mathcal{K}(X)$ is an M-ideal in $\mathcal{L}(X)$, then $X$ has the metric compact approximation property (in short MCAP) in general. For more details on M-ideal and its history, we refer the readers to a classic monograph \cite{HWW}. 

Among many properties involved in M-ideal theory, we focus on a property which we shall call the \textbf{adjoint compact perturbation property} (in short ACPP) and the \textbf{weak maximizing property} (in short WMP) for a pair $(X,Y)$ of Banach spaces $X$ and $Y$. We say that $(X,Y)$ has the ACPP if every operator $T\in\mathcal{L}(X,Y)$ whose adjoint $T^*$ does not attain its norm satisfies $\lVert T\rVert_e=\lVert T\rVert$ where $\lVert T\rVert_e=\inf_{K\in\mathcal{K}(X,Y)}\lVert T-K\rVert$. The ACPP is considered by W. Werner in \cite{WWerner} without its name, and it is proved that if $\mathcal{K}(X,Y)$ is an $M$-ideal in $\mathcal{L}(X,Y)$, then $(X,Y)$ has the ACPP in the same paper. Recently, R.M. Aron, D. Garc\'ia, D. Pellegrino and E.V. Teixeira introduced the WMP (see \cite{AGPT}). For an operator $T\in \mathcal{L}(X,Y)$, we say that a sequence $(x_n)\subset B_X$ is a \textbf{maximizing sequence} of $T$ if $\lim_n\lVert Tx_n\rVert=\lVert T\rVert$, and that $(X,Y)$ has the WMP if every $T\in \mathcal{L}(X,Y)$ having a non-weakly null maximizing sequence attains its norm. One of observations in \cite{AGPT} is that whenever $(X,Y)$ has the WMP if operators $T\in\mathcal{L}(X,Y)$ and $K\in\mathcal{K}(X,Y)$ satisfy $\lVert T+K\rVert>\lVert T\rVert$, then  $T+K$ attains its norm. Afterwards, the later property is named as the \textbf{compact perturbation property} (in short CPP) in \cite{GMA}. Indeed, the WMP implies the CPP for a pair of Banach spaces. 

It is clear that the CPP of a pair $(X,Y)$ is equivalent to a property that every $T\in\mathcal{L}(X,Y)$ which does not attain its norm satisfies $\lVert T\rVert_e=\lVert T\rVert$. This is why we called the ACPP that way. It is worth to comment that if $(X,Y)$ has the CPP, then $X$ is reflexive (see \cite{AGPT}). Since, whenever $X$ is reflexive, $T\in\mathcal{L}(X,Y)$ attains its norm if and only if $T^*$ attains its norm, we see that the CPP implies the ACPP.

\begin{center}
\begin{tikzcd}
 (X,Y) \text{~has the WMP} \arrow[rr, Rightarrow]     &    & (X,Y) \text{~has the CPP} \arrow[dd, Rightarrow, bend right]            \\
                                &    &                                                    \\
\mathcal{K}(X,Y) \text{~is an M-ideal} \arrow[rr, Rightarrow] &    & (X,Y) \text{~has the ACPP} \arrow[uu, "\text{if }X\text{ is Reflexive}"', bend right]
\end{tikzcd}
\end{center}

It is quite natural to ask whether or not the converse of each implication holds, and some of them are already known. First of all, note that $\mathcal{K}(Z,c_0)$ is an M-ideal in $\mathcal{L}(Z,c_0)$ for arbitrary Banach space $Z$ (see \cite[Example 4.1, Ch VI.4]{HWW}). Hence, if a Banach space $Z$ is not reflexive, then $(Z,c_0)$ has the ACPP but it does not have the CPP. Second, it is known that $(\mathbb{R}\oplus_\infty \ell_2,c_0)$ does not have the WMP (see \cite[Remark 3.2]{GP}), and this shows that neither the condition that $(X,Y)$ has the CPP  nor the one that $\mathcal{K}(X,Y)$ is an M-ideal implies the WMP. Therefore, we focus on the rest in the present paper, and we prove that none of them are true in general.

Nevertheless, we shall see that all the properties are closely related. In section \ref{section2}, we show that $(\ell_1,\ell_1)$ does not have the ACPP which is a strengthening of the well known fact that $\mathcal{K}(\ell_1)$ is not an M-ideal in $\mathcal{L}(\ell_1)$ (see \cite{L1} for instance). We also deduce that whenever $Y$ is an infinite dimensional Banach space without an isomorphic copy of $\ell_1$ if $(\ell_1,Y)$ has the ACPP, then $Y$ has the \textbf{local diametral $2$ property} (in short LD2P) that means every slice of $B_Y$ has diameter $2$. As a corollary, we deduce that if $Y$ is an infinite dimensional Banach space such that $\mathcal{K}(\ell_1,Y)$ is an M-ideal in $\mathcal{L}(\ell_1,Y)$, then $Y$ has the LD2P since those space $Y$ contains no isomorphic copy of $\ell_1$ (see \cite{LORW}). With the ACPP, we also generalize \cite[Theorem 2.12 (a)]{LORW} which states that if $Y$ is an infinite dimensional Banach space such that $\mathcal{K}(\ell_1,Y)$ is an M-ideal in $\mathcal{L}(\ell_1,Y)$, then every infinite dimensional subspace of a quotient of $Y$ is non-reflexive. Note that if $Y$ is reflexive, then it does not have the LD2P since it has the Radon-Nikod\'ym property.

In section \ref{section3}, we see that property $(M)$ is a commonality between M-ideals of compact operators and the WMP. Among many variations of property $(M)$,  we recall the one for contractive operators introduced by N. J. Kalton and D. Werner in \cite{KW}. For Banach spaces $X$ and $Y$, we say that a contraction $T\in \mathcal{L}(X,Y)$ has \textbf{property $\bf (M)$} if elements $x\in X$ and $y\in Y$ satisfy $\lVert y\rVert \leq \lVert x\rVert$, then
 $$\limsup_n\lVert y+Tx_n\rVert \leq \limsup_n\lVert x+x_n\rVert$$
for every weakly null sequence $(x_n) \subset X$. We also say that a pair $(X,Y)$ has property $(M)$ if every contraction in $\mathcal{L}(X,Y)$ has property $(M)$. One of main results in \cite{KW} is that whenever there is a sequence of compact operators $(K_n)$ on a separable space Banach space $X$ such that $K_n$ and $K^*_n$ converge strongly to $id_X$ and $id_{X^*}$ respectively and satisfies $\lim_n \|id_X-2K_n\|=1$, $(X,Y)$ has property $(M)$ if and only if $\mathcal{K}(X,Y)$ is an M-ideal in $\mathcal{L}(X,Y)$. As an application of this result, the authors found many examples of pairs $(X,Y)$, like $(\ell_p,L_p[0,1])$ for every $p\in [2,\infty)$, such that $\mathcal{K}(X,Y)$ is an M-ideal in $\mathcal{L}(X,Y)$. We observe in section \ref{section3} that if $\mathcal{K}(\ell_p,Y)$ is an M-ideal in $\mathcal{L}(\ell_p,Y)$ for $p\in (1,\infty)$, then $(\ell_p,Y)$ has the WMP. In order to prove this, we introduce property strict $(M)$ for operators and a pair $(X,Y)$ which is motivated by property $(M)$. The main result is that if $(X,Y)$ has property strict $(M)$, then $(X,Y)$ has the WMP, and this result not only covers known examples having the WMP but also produces new ones. We also  show that the converse of this result holds when $Y$ is $c_0$. 

In section \ref{section4}, we investigate the \textbf{weak* to weak* maximizing property} (in short weak*-to-weak*MP) (see \cite{GP}) for a pair $(Y^*, X^*)$ that means if  an adjoint operator $T^*$ of $T\in \mathcal{L}(X,Y)$ has a non-weakly* null maximizing sequence $(y^*_n)$ in $B_{Y^*}$, then $T^*$ attains its norm. Note  that if $(Y^*,X^*)$ has the weak*-to-weak*MP, then $(X,Y)$ has the ACPP (see \cite[Remark 1.3]{GP}). In order to do it, we consider properties $(sM^*)$, strict $(sM^*)$ and $(sO^*)$ which are the net versions on dual spaces with respect to the weak* topology of properties $(M)$, strict $(M)$ and $(O)$ that are considered in section \ref{section3}. We prove that if $(X,Y)$ has property strict $(sM^*)$, then $(Y^*,X^*)$ has the weak*-to-weak*MP, and if $(X,Y)$ has the WMP, then a pair $(Y^*,X^*)$ has the weak*-to-weak*MP whenever $(X,Y)$ has property $(sO^*)$. We comment that a question raised in \cite{GP} whether  if $(X,Y)$ has the WMP, then a pair $(Y^*,X^*)$ has the weak*-to-weak*MP or not is still not known. We also find new examples of pairs having the weak*-to-weak*MP.

In section \ref{section5}, we investigate a Banach space $X$ such that $\bar{\delta}_X=\bar{\rho}_X$ where $\bar{\delta}_X$ and $\bar{\rho}_X$ are the moduli of asymptotic uniform convexity and smoothness, respectively. In order to do this, we introduce net-stable Banach spaces and weak-net-stable Banach spaces which are  motivated by stable Banach spaces introduced by J.L. Krivine and B. Maurey in \cite{KM}, and we prove that if $X$ is a weak-net-stable Banach space with property $(sM)$, the net version of property $(M)$, then $\bar{\delta}_X=\bar{\rho}_X$. The converse holds when $\bar{\rho}_X$ induces a symmetric norm on $\R^2$. Using this, we revisit the famous characterization of $X$ such that $\mathcal{K}(X\oplus_p X)$ ($1<p<\infty$) is an M-ideal given by N. J. Kalton in \cite{K1} and E. Oja in \cite{Oja1}. We also characterize the net version of the Opial property in terms of $\bar{\delta}_X$. 

\section{The adjoint compact perturbation property}\label{section2}

We recall a result of W. Werner which serves as the main motivation.
\begin{theorem}\cite{WWerner}\label{ACPPMid} If $\mathcal{K}(X,Y)$ is an $M$-ideal in $\mathcal{L}(X,Y)$, then  $(X,Y)$ has the ACPP.
\end{theorem}

The main purpose of the present section is to observe some properties of a pair $(X,Y)$ having the ACPP which can be observed in a pair $(X,Y)$ with a condition that $\mathcal{K}(X,Y)$ is an $M$-ideal in $\mathcal{L}(X,Y)$. We first see that the converse of theorem \ref{ACPPMid} does not hold in general using I-polyhedrality.

As convex polytopes in a finite dimensional space had been considered in many areas, there had been many attempts to consider infinite-dimensional convex polytopes. Among various notions of polyhedrality (see \cite{VL,GA}), we recall \textbf{I-polyhedrality}. We say that a Banach space $X$ is I-polyhedral if it satisfies
$$(Ext B_{X^*} )' \subset \{0\}$$
where $Ext  B_{X^*}$ is the set of extreme points of a set $B_{X^*}$ and $(Ext B_{X^*})'$ is the set of weak* cluster points of $Ext B_{X^*}$.
It is worth to note that $X$ is I-polyhedral if and only if $X$ is isometric to a subspace of $c_0 (\Gamma)$ where $\text{card~} \Gamma = \text{dens~} X$ and $\text{dens~} X$ is the smallest cardinality of a dense subset of $X$ (see  \cite[Theorem 1.2.]{VL}).

\begin{theorem}
If a Banach space $Y$ is I-polyhedral, then a pair $(X,Y)$ has the ACPP for every Banach space $X$.
\begin{proof}
Assume that the adjoint $T^*$ of an operator $T \in\mathcal{L}(X,Y)$ does not attain its norm. Since it holds that 
$$\lVert T^* \rVert=\sup_{v^* \in Ext B_{Y^*}}{\lVert T^* v^* \rVert},$$
there is a sequence $(v_n ^*) \subset Ext B_{Y^*}$ such that $\left(\|T^* v_n^*\|\right)$ converges to $\|T^*\|$.

From the weak* compactness of $B_{Y^*}$, take a subnet $(v^*_{n_\alpha})$ of $(v^*_n)$ which converges to some $v\in B_X$ with respect to the weak* topology, and we see that $v=0$. Indeed, there exists $\alpha_0$ such that $v\notin \{v^*_{n_\alpha} : \alpha\geq \alpha_0\}$ since $T^*$ does not attain its norm. This gives us that $v$ is a weak* cluster point of $ Ext B_{Y^*}$ that means $v=0$. Hence, we see that $$\|T\|=\lim_\alpha \|(T+K)^*v^*_{n_\alpha}\|\leq \|T+K\|$$ for every $K\in \mathcal{K}(X,Y)$. Therefore, $\|T\|=\|T\|_e$.
\end{proof}
\end{theorem}

Since there is a subspace of $c_0$ without the MCAP (see \cite{LT}), we have a Banach space $Z$ which is I-polyhedral and $\mathcal{K}(Z,Z)$ is not an M-ideal in $\mathcal{L}(Z,Z)$ (see \cite[Ch. VI]{HWW}). Hence, this shows that the converse of theorem \ref{ACPPMid} is not true. 

In \cite{VL}, 8 types of polyhedrality  including I-polyhedrality are considered, and it is true that I-polyhedrality implies the others. Among them, we recall \textbf{II-polyhedrality}. We say that $X$ is II-polyhedral  if there exists $0<r<1$ such that 
$$(Ext B_{X^*} )' \subset r B_{X^*}.$$
Since this implies the other types of polyhedrality except I-polyhedrality, it is quite natural to ask whether or not a pair $(X,Y)$ has the ACPP for every Banach space $X$ whenever $Y$ is a II-polyhedral Banach space. We give a negative answer in example \ref{exfailACPP}.

 We now focus on results related to an M-ideal of compact operators. It is well known that $\mathcal{K}(\ell_1)$ is not an M-ideal in $\mathcal{L}(\ell_1)$ (see \cite{L1,SW}). We first strengthen this result with the ACPP.

\begin{prop}\label{thm:ell1}
The pair $(\ell_1,\ell_1)$ does not have the ACPP.
\end{prop}

\begin{proof}
Define a permutation $\phi_n$ on $\mathbb{N}$ and a function $\gamma_n :\mathbb{N} \to \{1,-1\}$ for each $n\in \mathbb{N}$ by
$$\phi_n=(1,n) \text{~and~}\gamma_n(m)=
\begin{cases}
1  &\mbox{ if } m\neq n \\
-1 &\mbox{ if } m=n 
\end{cases},$$
respectively. Consider an infinite matrix $(\alpha_{n,m})_{n,m}$ given by 
$$(\alpha_{n,m})_{n,m}=\left(\gamma_n(m)\left(1-\frac{1}{2^n}\right)\frac{1}{2^{\phi_n(m)}}\right)_{n,m}.$$

It is clear that $w_m^*=(\alpha_{n,m})_n$ belongs to $B_{\ell_\infty}$  for every $m\in\mathbb{N}$, and that
$$\sum_{m=1}^\infty \lvert w_m^*(e_n)\rvert =1-\frac{1}{2^n}$$
 for every $n\in \mathbb{N}$ where $(e_n)$ is the canonical basis of $\ell_1$. Hence, an operator $T=\sum_{m=1}^\infty w_m^* \otimes e_m$ belongs to $S_{\mathcal{L}(\ell_1)}$.

We now claim that adjoint $T^*$ of $T$ does not attain its norm. Otherwise, it is obvious that $T^*$ attains its norm at some $\delta=(\delta_n)_n\in S_{\ell_\infty}$ which is a sequence of numbers of modulus $1$. Indeed, it holds that
$$1=\lVert T^*\delta\rVert=\sup_{n}\left\lvert\sum_{m=1}^\infty \delta_m w^*_m(e_n)\right\rvert=\sup_{n}\left(1-\frac{1}{2^{n}}\right)\left\lvert\sum_{m=1}^\infty \delta_m \gamma_n(m)\frac{1}{2^{\phi_n(m)}}\right\rvert.$$
Therefore, there exists $n_0\in\mathbb{N}$ such that
$$\left(1-\frac{1}{2^{n_0}}\right)\left\lvert\sum_{m=1}^\infty \delta_m\gamma_{n_0}(m)\frac{1}{2^{\phi_{n_0}(m)}}\right\lvert>\frac{7}{8}.$$
It is clear that $n_0>3$ and $\delta_2\neq \delta_{n_0}$, and these give that $\left|\delta_2\frac{1}{4}+\delta_{n_0}\frac{1}{2^{n_0}}\right|<\frac{1}{4}+\frac{1}{2^{n_0}}$. Hence, there is $n_1\in\mathbb{N}$ such that
$$\left(1-\frac{1}{2^{n_1}}\right)\left\lvert\sum_{m=1}^\infty \delta_m\gamma_{n_1}(m)\frac{1}{2^{\phi_{n_1}(m)}}\right\lvert>\alpha$$
where $\alpha=\max\left\{1-\frac{1}{2^{n_0}},\left|\delta_2\frac{1}{4}+\delta_{n_0}\frac{1}{2^{n_0}}\right|+\sum_{m\in \mathbb{N}\setminus\{2,n_0\}}\frac{1}{2^m}\right\}$.
Note that $n_1>n_0$. Thus, we have
\begin{align*}
\alpha & <\left(1-\frac{1}{2^{n_1}}\right)\left\lvert\sum_{m=1}^\infty \delta_m\gamma_{n_1}(m)\frac{1}{2^{\phi_{n_1}(m)}}\right\rvert\\
&\leq \left(1-\frac{1}{2^{n_1}}\right)\left|\left(\delta_2\frac{1}{4}+\delta_{n_0}\frac{1}{2^{n_0}}\right)+\sum_{m\in \mathbb{N}\setminus\{2,n_0\}} \delta_m\gamma_{n_1}(m)\frac{1}{2^{\phi_{n_1}(m)}}\right|\\
&\leq \left(1-\frac{1}{2^{n_1}}\right)\left[\left|\delta_2\frac{1}{4}+\delta_{n_0}\frac{1}{2^{n_0}}\right|+\sum_{m\in \mathbb{N}\setminus\{2,n_0\}}\frac{1}{2^m}\right] \\
&\leq \left(1-\frac{1}{2^{n_1}}\right) \alpha
\end{align*}
which leads to a contradiction. On the other hand, we have 
$$\|T\|_e\leq \lVert T- w_2^*\otimes e_2 \rVert= \sup_{k}\sum_{m\in \mathbb{N}\setminus\{2\}}\lvert w_m^*(e_k)\rvert\leq \frac{3}{4} <1= \lVert T \rVert$$ 
which finishes the proof.
\end{proof}

We recall \cite[Theorem 2.12]{LORW} that if $Y$ is an infinite dimensional Banach space such that $\mathcal{K}(\ell_1,Y)$ is an M-ideal in $\mathcal{L}(\ell_1,Y)$, then every infinite dimensional subspace of a quotient of $Y$ is non-reflexive. We deduce the same result with the ACPP.
 
\begin{theorem}\label{thm:finite}
For  an infinite dimensional Banach space $Y$ containing no isomorphic copy of $\ell_1$, if the pair $(\ell_1,Y)$ has the ACPP, then $Y$ has the LD2P.
\begin{proof} If $Y$ does not have the LD2P, then there is a functional $y_0 ^* \in S_{Y^*}$ and $\delta>0$ such that  
$$\text{diam~}\left[S(y_0 ^* , \delta)\right] < 2$$
where $S(y_0 ^* , \delta)=\{y\in B_Y ~:~ \re y_0^*(y)>1-\delta\}$. From the Bishop-Phelps theorem, we assume that $y_0^*(y)=1$ for some $y \in S_Y$.

 Let $\alpha=diam\left[S(y_0 ^* , \delta)\right] $ and $X=\ker(y_0 ^*)$. Since $X$ contains no isomorphic copy of $\ell_1$, by the Rosenthal's $\ell_1$ theorem, there is a normalized weakly null sequence $(x_n)\subset X$.

For $\theta>0$ such that $\frac{1}{1+\theta}>1-\frac{\delta}{2}$, there exists $\delta_i \in \mathbb{R}$ such that 
$$\lvert \delta_i \rvert < 1+\theta - \frac{(1+\theta)(2-\alpha)}{2} \text{~and~} \lVert y + \delta_i x_i \rVert = 1+\theta.$$

 Indeed, for  $u_i > 0$ and $v_i<0$ such that $\lVert y+u_i x_i \rVert =\lVert y+v_i x_i \rVert=1+\theta$, if both $\lvert u_i \rvert$ and $\lvert v_i \rvert$ are greater or equal to $ 1+\theta - \frac{(1+\theta)(2-\alpha)}{2}$, then we have
$$\left\lVert \frac{1}{1+\theta}(y+u_i x_i) - \frac{1}{1+\theta}(y+v_i x_i) \right\rVert = \frac{\lvert u_i \rvert +\lvert v_i \rvert}{1+\theta} \geq \alpha.$$
Since $\frac{1}{1+\theta}(y+u_i x_i)$ and $\frac{1}{1+\theta}(y+v_i x_i)$ are elements of $S(y_0 ^*,\delta)$, this leads to  a contradiction.

Define $T$ and $U$ in $\mathcal{L}(\ell_1,Y)$ by $Te_i = y+\delta_i \left( 1- \frac{1}{i} \right) x_i$ and $Ue_i=y$ where $(e_i)$ is the canonical basis of $\ell_1$. Since $y+\delta_i \left( 1- \frac{1}{i} \right) x_i$ is a convex combinations of $y$ and $y+\delta_i x_i$ for each $i\in \mathbb{N}$, we have 
$$\left\lVert y+\delta_i \left( 1- \frac{1}{i} \right) x_i \right\rVert\leq 1+\theta.$$

 Therefore, we have $\lVert T \rVert = 1+\theta$ and $$\lVert T \rVert_e  \leq \|T-U\|\leq 1+\theta- \frac{(1+\theta)(2-\alpha)}{2}<\|T\|.$$

We now claim that $T^*$ does not attains its norm. Otherwise, $T^*$ attains it's norm at some $y^* \in S_{Y^*}$. Since we have
$$\left<T^*y^* , e_i \right>= \left<y^* , Te_i \right>=\left<y^* , y+\delta_i \left( 1- \frac{1}{i} \right) x_i \right>=y^* (y) +\delta_i \left( 1-\frac{1}{i} \right) y^*(x_i) $$
for all $i \in \mathbb{N}$, we see that 
$$T^* (y^*)=\left(y^*(y) +\delta_i \left( 1-\frac{1}{i} \right) y^*(x_i) \right)_i.$$
The fact that $(x_i)$ is weakly null implies that there exists $i_0\in \mathbb{N}$ such that $$\left\lVert y^*(y) +\delta_i \left( 1-\frac{1}{i} \right) y^*(x_i)\right\rVert \leq 1+ \frac{\theta}{2}<1+\theta$$
for every $i>i_0$. Hence, there is a natural number $j\leq i_0$ such that 
$$\left\lVert y^*(y) +\delta_j \left( 1-\frac{1}{j} \right) y^*(x_j)\right\rVert  = 1+ \theta.$$
On the other hand, we have 
\begin{align*}
\left|y^*(y) +\delta_j \left( 1-\frac{1}{j} \right) y^*(x_j)\right|
&=\left\lvert\left(1-\frac{1}{j}\right) \left(y^*(y) +\delta_j  y^*(x_j)\right) +\frac{1}{j}y^* (y) \right\rVert \\
&\leq \left(1-\frac{1}{j}\right)\left\lVert y+\delta_j  x_j \right\rVert+\frac{1}{j}\|y\|\\
&\leq \left(1-\frac{1}{j}\right)(1+\theta)+\frac{1}{j}\\
&<1+\theta
\end{align*}
which leads to a contradiction.
\end{proof}
\end{theorem}

We present the following lemma to generalize proposition \ref{thm:ell1} and theorem \ref{thm:finite}.

\begin{lemma}\label{lem:inherite}
For Banach spaces $X$ and $Y$ and their closed subspaces $X_1 \subset X$ and $Y_1 \subset Y$, if the pair $(X,Y)$ has the ACPP, then the pair $(X/X_1,Y_1)$ has the ACPP.
\end{lemma}
\begin{proof}
Let $\pi:X\to X/{X_1}$ be the natural quotient map and $i:Y_1\to Y$ be the natural injection. For an operator $T\in\mathcal{L}(X/{X_1},Y_1)$, if its adjoint operator $T^*$ does not attain its norm, then neither does $\pi^*\circ T^* \circ i^*\in\mathcal{L}(Y^*,X^*)$. Otherwise, there is $y^*\in S_{Y^*}$ such that $\|\pi^*\circ T^* \circ i^*\|=\|\left(\pi^*\circ T^* \circ i^*\right)y^*\|$, and $T^*$ attains its norm at $i^*y^*\in S_{Y_1^*}$ since $\|T^*\|=\|\pi^*\circ T^* \circ i^*\|$.

 Since $(X,Y)$ has the ACPP, we have $$\|i\circ T\circ\pi\|_e = \|i\circ T\circ\pi\|=\|T\|.$$

Hence, we see that 
$$\|T-K\|=\|i\circ (T-K)\circ\pi\|=\|i\circ T\circ \pi-i\circ K\circ \pi\| \geq \|i\circ T\circ \pi\|=\|T\|$$ for any $K\in\mathcal{K}(X/{X_1},Y_1)$ which shows that $\|T\|_e=\|T\|$. Therefore, $(X/X_1,Y_1)$ has the ACPP.
\end{proof}

\begin{cor} For an infinite dimensional Banach space $Y$,
\begin{enumerate}
\item if $\ell_1$ is a subspace of $Y$, then the pair $(\ell_1,Y)$ does not have the ACPP.
\item if $Y$ contains no isomorphic copy of $\ell_1$ and $(\ell_1,Y)$ has the ACPP, then every infinite-dimensional subspace of $Y$ has the LD2P.
\end{enumerate}
\end{cor}

 Remark that if $\mathcal{K}(\ell_1,Y)$ is an M-ideal in $\mathcal{L}(\ell_1,Y)$, then $Y$ is an Asplund space \cite[Lemma 2.6]{LORW}. Hence, we have the following consequence of theorem \ref{thm:finite} which generalizes \cite[Theorem 2.12 (a)]{LORW}.

\begin{cor}\label{midld2p}
For an infinite dimensional Banach space $Y$, if $\mathcal{K}(\ell_1,Y)$ is an M-ideal in $\mathcal{L}(\ell_1,Y)$, then every infinite dimensional subspace of a quotient of $Y$ has the LD2P. In particular, this subspace is not reflexive.
\end{cor}
\begin{proof} If $\mathcal{K}(\ell_1,Y)$ is an M-ideal in $\mathcal{L}(\ell_1,Y)$, then $\mathcal{K}(\ell_1,Y/Z)$ is an M-ideal in $\mathcal{L}(\ell_1,Y/Z)$ for arbitrary closed subspace $Z$ of $Y$ by \cite[Proposition 2.11]{LORW}. Hence, we have that $(\ell_1,Y/Z)$ has the ACPP by theorem \ref{ACPPMid} and $Y/Z$ does not contain an isomorphic copy of $\ell_1$ by \cite[Lemma 2.6]{LORW}. Therefore, we deduce that every subspace of $Y/Z$ has the LD2P if it is infinite dimensional. Moreover, it is not reflexive since every reflexive space has the Radon-Nikod\'ym property that means its closed unit ball contains a slice with arbitrary diameter $\varepsilon>0$.
\end{proof}

 In several papers, a Banach space $X$ satisfying that $\mathcal{K}(X\oplus_p X)$ is an M-ideal in $\mathcal{L}(X\oplus_p X)$ for $1<p\leq \infty$ is studied. This property is called \textbf{property $\bf (M_p)$}, and the following characterization of a Banach space $X$ having property $(M_\infty)$ is well known.

\begin{theorem}\cite[Corollary 6,7]{Oja1}\label{Ojaminf} For a Banach space $Y$, the following are equivalent.
\begin{enumerate}[(i)]
\item $Y$ has property $(M_\infty)$.
\item $Y$ has the MCAP and $\mathcal{K}(\ell_1,Y)$ is an M-ideal in $\mathcal{L}(\ell_1,Y)$.
\item $Y$ has the MCAP, contains no isomorphic copy of $\ell_1$, and has \textbf{property $\bf (m_\infty)$} which means that for every weakly null sequence $(y_n)\subset Y$ and $y\in Y$
$$\limsup_{n} \| y + y_n\|=\max\left\{\|y\|, \limsup_{n}  \|y_n\|\right\}.$$ 
\end{enumerate}
\end{theorem}

It is worth to introduce a conjecture in \cite{LORW} that if  $\mathcal{K}(\ell_1,Y)$ is an M-ideal in $\mathcal{L}(\ell_1,Y)$, then $Y$ has property $(m_\infty)$. We still don't know the answer, but we guess that it could be true since if an infinite dimensional Banach space having property $(m_\infty)$ contains no isomorphic copy of $\ell_1$ then it has the LD2P.

It is known that if a pair $(X,\ell_\infty)$ has the CPP, then $X$ is finite-dimensional (\cite[Theorem 2.1]{GMA}). It is worth to comment that if the pair $(X,\ell_\infty)$  has the CPP, then the pair $(\ell_1,X^*)$ has the ACPP by the following theorem. Hence, we get \cite[Theorem 2.1]{GMA} as a corollary of theorem \ref{thm:finite}. 

\begin{theorem}For Banach spaces $X$ and $Y$, if  the pair $(Y^*,X^*)$ has the CPP, then the pair $(X,Y)$ has the ACPP.
\end{theorem}
\begin{proof} Assume that the adjoint $T^*$ of $T\in \mathcal{L}(X,Y)$ does not attain its norm. Since $(Y^*,X^*)$ has the CPP, we see that $\|T^*\|_e=\|T^*\|$. Hence, the statement is followed by the inequality 
$$\|T^*\|_e\leq \|T\|_e\leq \|T\|=\|T^*\|$$
which is obvious.
\end{proof}

As an application of theorem \ref{thm:finite}, we find a II-polyhedral Banach space $Y$ such that $(\ell_1,Y)$ does not have the ACPP.

\begin{example}\label{exfailACPP}
There exists a II-polyhedral renorming $Y$ of $c_0$ such that $Y$ does not have the ACPP. 
\end{example}
\begin{proof}
For $X= c_0 \oplus_1 \mathbb{R}$, let $Y= c_0 \oplus \mathbb{R}$ be equipped with the norm such that $B_{Y^*} = \text{conv}( B_{X^*} \cup \{(0, \pm2)\})$.  This is the one given in \cite[Example 4.1]{VL} as an example of II-polyhedral Banach spaces which is not I-polyhedral. It is obvious that $Y$ does not have the LD2P since $B_Y$ is the convex hull of $B_{c_0} \times \{0\} \cup \frac{1}{2}B_{c_0} \times \left\{\pm \frac{1}{2} \right\}$.
\end{proof}

We finish this section with giving an analogue of \cite[Proposition 3.2]{GMA} for ACPP which presents a class of Banach spaces which does not have ACPP.

\begin{prop}
For $1 \leq q < p \leq \infty$ and Banach spaces $X$ and $Y$, if there exists $T \in \mathcal{L}(X,Y)$ whose adjoint $T^*$ does not attain its norm, then the pair $(X \oplus_p \mathbb{R},Y\oplus_q \mathbb{R})$ fails the ACPP.
\end{prop}
\begin{proof}
For the convenience, we assume that $\lVert T \rVert =1$, and define $S \in \mathcal{L}(X \oplus_p \mathbb{R},Y\oplus_q \mathbb{R})$ by $S(x,a)=(Tx,a)$ for each $(x,a)\in X\oplus_p \mathbb{R}$. 

For the conjugates $p^*$ and $q^*$ of $p$ and $q$ respectively, if $q \neq 1$, then we have

\begin{align*}
 \lVert S^* \rVert 
&= \sup_{\substack{y^* \in S_{Y^*} \\ 0 \leq \lambda \leq 1}}{\left\lVert S^* \left(\lambda y^* , \pm \left(1-\lambda^{q^*} \right)^{\frac{1}{q^*}}\right) \right\rVert_{p^*}} \\
&=\sup_{0 \leq \lambda \leq 1 }{\left(\lambda^{p^*}\|T^*\|^{p^*}+\left(1-\lambda^{q^*}\right)^{\frac{p^*}{q^*}}\right)^{\frac{1}{p^*}}}\\
&=\sup_{0 \leq \lambda \leq 1 }\|(\lambda,(1-\lambda^{q^*})^{1/q^*}\|_{p^*}>1.
\end{align*}

If $q=1$, then we have 
$$\lVert S^* \rVert = \sup_{y^* \in S_{Y^*}}{\left\lVert S^* (y^*, \pm 1) \right\rVert_{p^*} } 
=\sup_{y^* \in S_{Y^*}}{\left(\lVert T^*y^* \rVert^{p^*} +1 \right)^{\frac{1}{p^*}}}  
= 2^{\frac{1}{p^*}} >1 
$$
These show that $S^*$ does not attain its norm since $T^*$ attains its norm otherwise. Moreover, we also have
$$\|T\oplus 0 + 0\oplus id_\mathbb{R}\|=\lVert S \rVert>\|T\|=\|T\oplus 0 \|\geq \|S\|_e.$$
\end{proof}
\begin{cor}
For $1 \leq q < p \leq \infty$ and Banach spaces $X$ and $Y$, if there exists $T \in \mathcal{L}(X,Y)$ whose adjoint $T^*$ does not attain its norm, then the pair $\mathcal{K}(X \oplus_p \mathbb{R},Y\oplus_q \mathbb{R})$ is not an M-ideal in $\mathcal{L}(X \oplus_p \mathbb{R},Y\oplus_q \mathbb{R})$.
\end{cor}
\section{Property $(M)$ and the weak maximizing property}\label{section3}

Recently, R.M. Aron, D. Garc\'ia, D. Pellegrino and E.V. Teixeira introduced the WMP for a pair $(X, Y)$ of Banach spaces $X$ and $Y$, and proved that if $(X,Y)$ has the WMP, then it has the CPP and $X$ is reflexive (see \cite{AGPT}). The most typical examples of pairs having the WMP are the pair $(H,H)$ of a Hilbert space $H$ and the pair $(\ell_p,\ell_q)$ for $p,q\in(1,\infty)$. We refer the readers to \cite{DJM,GP,GMA} for more recent results on the WMP and the CPP.

Motivated by the fact that the CPP implies the ACPP, we ask the following questions:

\begin{itemize}
\item Is $\mathcal{K}(X,Y)$ an M-ideal in $\mathcal{L}(X,Y)$ if $(X,Y)$ has the WMP?

\item Does $(X,Y)$ have the WMP if $\mathcal{K}(X,Y)$ is an M-ideal in $\mathcal{L}(X,Y)$ and $X$ is reflexive?
\end{itemize}

To give an answer for the later question, we consider $(\mathbb{R}\oplus_\infty \ell_2,c_0)$. According to \cite[Remark 3.2]{GP}, this pair does not have the WMP. However, $\mathcal{K}(\mathbb{R}\oplus_\infty \ell_2,c_0)$ is an M-ideal in $\mathcal{L}(\mathbb{R}\oplus_\infty \ell_2,c_0)$. In the present section, we give a  negative answer for the other question. Nevertheless, we obtain some results on the WMP that are related to property $(M)$ and M-ideal of compact operators.

Among many variations of property $(M)$, we use the one for contractions introduced by N. J. Kalton and D. Werner in \cite{KW} (see section \ref{section1} for its definition). Note that $(X,X)$ has property $(M)$ if and only if $X$ has property $(M)$ that means $id_X$ has property $(M)$ (see \cite{K1}). The authors employed this property to find a pair $(X,Y)$ such that $\mathcal{K}(X,Y)$ is an M-ideal in $\mathcal{L}(X,Y)$, and they proved, for example, that $\mathcal{K}(\ell_p,L_p[0,1])$ is an M-ideal in $\mathcal{L}(\ell_p,L_p[0,1])$ for $2\leq p <\infty$. 

Motivated by property $(M)$, we introduce property \textbf{strict $\bf (M)$} to study the WMP. For Banach spaces $X$ and $Y$, we shall say that a contraction $T\in \mathcal{L}(X,Y)$ has property strict $(M)$ if elements $x\in X$ and $y\in Y$ satisfy $\lVert y\rVert < \lVert x\rVert$, then
 $$\limsup_n\lVert y+Tx_n\rVert < \limsup_n\lVert x+x_n\rVert$$
for every weakly null sequence $(x_n) \subset X$. We also say that a pair $(X,Y)$ has property strict $(M)$ if every contraction $T\in \mathcal{L}(X,Y)$ has property strict $(M)$. 

\begin{theorem}\label{MWMP}
For Banach spaces $X$ and $Y$, if $X$ is reflexive and $(X,Y)$ has property strict $(M)$, then $(X,Y)$ has the WMP.
\end{theorem}
\begin{proof}
For $T\in \mathcal{L}(X,Y)$, let $(x_n) \subset B_X$ be a non-weakly null maximizing sequence of $T$. Without loss of generality, we assume that $\|T\|=1$. Passing to a subsequence, we also assume that $x_n$ converges weakly to a non-zero element $x$ of $B_X$, and define a weakly null sequence $(w_n)$  by $w_n = x_n -x$.

If $T$ does not attain its norm, then $\lVert Tx \rVert < \lVert x \rVert$. Hence, we have
$$1=\limsup_{n}{\lVert T x_n \rVert} = \limsup_n {\lVert Tx + Tw_n\rVert}<\limsup_n{\lVert x + w_n \rVert }=\limsup_n{\lVert x_n \rVert} \leq 1$$
which leads to a contradiction.
\end{proof}

We do not know whether the converse of theorem \ref{MWMP} is true or not. However, it holds if $Y=c_0$.
\begin{theorem}\label{thmc0}
For  a reflexive Banach space $X$, $(X,c_0)$ has the WMP if and only if it has property strict $(M)$.
\end{theorem}
\begin{proof} From theorem \ref{MWMP}, it is enough to prove the necessity. Indeed, assume that $(X,c_0)$ has the WMP. We first claim that every contraction which does not attain its norm has property strict $(M)$. If not, there are elements $x\in X$ and $y\in c_0$ with $\lVert y\rVert<\lVert x\rVert$, a contraction $T\in \mathcal{L}(X,c_0)$ which does not attain its norm and a weakly null sequence $(x_n)$ in $X$ such that
$$\limsup_n\lVert y+Tx_n\rVert \geq \limsup_n\lVert x+x_n\rVert.$$
Passing to a subsequence, we assume that $\lim_n \|y+Tx_n\|=\limsup_n\lVert y+Tx_n\rVert$ and $\lim_n\|x+x_n\|=\limsup_n\|x+x_n\|$. 

We note that for any $z\in c_0$
$$\lim_n\|z+Tx_n\|=\left\|\left(\|z\|,\limsup_n\|Tx_n\|\right)\right\|_\infty.$$

 If $\lim_n\lVert y+Tx_n\rVert=\lVert y\rVert$, then we have
\begin{equation}\label{eqnc0}
\lim_n\lVert y+Tx_n\rVert=\lVert y\rVert<\lVert x\rVert\leq \lim_n\lVert x+x_n\rVert
\end{equation}
which leads to a contradiction. Hence, we see $\lim_n\lVert y+Tx_n\rVert=\limsup_n\lVert Tx_n\rVert$, and this shows that
$$\lim_n\lVert x+x_n\rVert\leq\lim_n\lVert y+Tx_n\rVert=\limsup_n\lVert Tx_n\rVert\leq \limsup_n\lVert Tx+Tx_n\rVert\leq\lim_n\lVert x+x_n\rVert.$$
Passing to a subsequence, we see that $\frac{x+x_n}{\lVert x+x_n\rVert}$ is a non-weakly null maximizing sequence of $T$. From the WMP, we get that $T$ attains its norm which leads to a contradiction.

 We now show that every contraction has property strict $(M)$. Similarly to the above, assume that there is a contraction $T\in \mathcal{L}(X,c_0)\setminus\{0\}$, elements $x\in X$ and $y\in c_0$ with $\lVert y\rVert<\lVert x\rVert$ and a weakly null sequence $(x_n)$ in $X$ such that both limits $\lim_n\lVert y+Tx_n\rVert$ and $\lim_n\lVert x+x_n\rVert$ exist and
$$\lim_n\lVert y+Tx_n\rVert \geq \lim_n\lVert x+x_n\rVert.$$
As is the case with \eqref{eqnc0}, we have
$$
\limsup_n\lVert Tx_n\rVert\geq\lim_n\lVert x+x_n\rVert
$$
and this shows that 
$$
\limsup_n\lVert x_n\rVert\geq \limsup_n\lVert Tx_n\rVert\geq\|x\|>0.
$$

It is obvious that $T$ is of the form $T=\sum_{i=1}^\infty x_i^*\otimes e_i$ for a weakly* null sequence  $(x_i^*)$ in $X^*$ such that $\sup_i\lVert x_i^*\rVert=\lVert T\rVert$ where $(e_i)$ is the canonical basis of $c_0$.

For a set $A=\left\{i\in \mathbb{N}~:~\|x_i^*\|<\frac{\|x\|}{2\limsup_i\|x_i\|}\right\}$, we define contractions $T_1$ and $T_2$ in $\mathcal{L}(X,c_0)$ by 
$$T_1=\sum_{i\in A}  x_i^*\otimes e_i \text{~and~} T_2=\sum_{i\notin A}  x_i^*\otimes e_i.$$

It is clear that $T=T_1+T_2$ and $\|Tz\|=\left\|\left(\|T_1z\|,\|T_2z\|\right)\right\|$ for every $z\in X$, and these show that 
$$\limsup_n \|Tx_n\|=\limsup_n\|T_2x_n\|$$
 since $\limsup_n \|T_1x_n\|<\|x\|$. Define $S\in \mathcal{L}(X,c_0)$ by 
$$
S=\sum_{i\notin A}  \left(1-\frac{1}{i}\right)\frac{1}{\|x_i^*\|}x_i^*\otimes e_i.
$$
We list a few facts that are obvious.
\begin{enumerate}[(a)]
\item A set $A^c$ is infinite.
\item A sequence $\left(\left(1-\frac{1}{i}\right)\frac{1}{\|x_i^*\|}x_i^*\right)$ is weakly* null.
\item An operator $S$ does not attain its norm and $\|S\|=1$.
\end{enumerate}

Hence, to get a desired contradiction, it is enough to show that $\limsup_n \|Sx_n\|\geq \limsup_n \|T_2x_n\|$ since $S$ is a contraction which does not attain its norm and not have property strict $(M)$ if it is true.

 For arbitrary $\varepsilon>0$, take $i_0\in A^c$ such that $ \frac{1}{i_0} <\varepsilon$ and $n_0\in \mathbb{N}$ such that $\left\|\left(\sum_{i\leq i_0}  x_i^*\otimes e_i\right)(x_n)\right\|<\varepsilon$ for every $n\geq n_0$.
Then, for every $n\geq n_0$, we see that
\begin{align*}
\|Sx_n\|
&= \sup_{i\in A^c} \left|\left(1-\frac{1}{i}\right)\frac{1}{\|x_i^*\|}x^*_i(x_n)\right|\\
&\geq (1-\varepsilon)\sup_{i\in A^c,~i> i_0} \left|x^*_i(x_n)\right| \\
&\geq (1-\varepsilon)(\|T_2x_n\|-\varepsilon)\\
&\geq \|T_2x_n\|-\varepsilon\|T\|\sup_m\|x_m\|-(1-\varepsilon)\varepsilon
\end{align*}
which finishes the proof.
\end{proof}
 Similarly to the case with property $(M)$, property strict $(M)$ can be defined for a Banach space. We say that a Banach space $X$ has property strict $(M)$ if $id_X$ has property strict $(M)$, equivalently if elements $x\in X$ and $y\in Y$ satisfy $\lVert y\rVert < \lVert x\rVert$, then
 $$\limsup_n\lVert y+x_n\rVert < \limsup_n\lVert x+x_n\rVert$$
for every weakly null sequence $(x_n) \subset X$. 

To investigate property strict $(M)$, we recall the \textbf{Opial property}. We say that a Banach space $X$ has the Opial property if every non-zero element $x\in X$ and every weakly null sequence $(x_n)$ in $X$ satisfy 
$$\limsup_n\lVert x_n\rVert <\limsup_n\lVert x+x_n\rVert.$$

Z. Opial \cite{Op} showed that every Hilbert space has the Opial property to study the weak convergence of $(T^nx)$ for a nonexpansive selfmapping $T$ of a closed convex subset $C$ and $x\in C$. There has been intensive study in this direction, and we mention a result of D. van Dulst that every  separable Banach space can be renormed to have the Opial property \cite{Dul}.

As is the case with property $(M)$, we introduce the operator version of the Opial property, and we shall say \textbf{property $\bf (O)$} for the convenience. A contraction $T\in\mathcal{L}(X,Y)$ has property $(O)$ if every non-zero element $x\in X$ and every weakly null sequence $(x_n)$ in $X$ satisfy
$$\limsup_n\lVert Tx_n\rVert<\limsup_n\lVert x+x_n\rVert.$$
We also say that a pair $(X,Y)$ has property $(O)$ if every contraction $T\in\mathcal{L}(X,Y)$ has property  $(O)$.

We see the characterization of property strict $(M)$ in terms of properties $(M)$ and $(O)$. 

\begin{lemma}\label{sp}
For Banach spaces $X$ and $Y$, a contraction $T\in\mathcal{L}(X,Y)$ has property strict $(M)$ if and only if it has both properties $(M)$ and $(O)$. In particular,  $(X,Y)$ has property strict $(M)$ if and only if it has both properties $(M)$ and $(O)$.
\end{lemma}
\begin{proof}
We first prove the necessity. Since it is obvious that property strict $(M)$ implies property $(O)$, we only prove that $T$ has property $(M)$ if $T$ has property strict $(M)$. Let elements $x\in X$ and $y\in Y$ satisfy that $\lVert y\rVert\leq \lVert x\rVert$ and $(x_n)$ be a weakly null sequence in $X$. 

For arbitrary $\lambda\in(0,1)$, we have that
$$\limsup_n\lVert \lambda y+Tx_n\rVert<\limsup_n\lVert x+x_n\rVert,$$
which shows that $T$ has property $(M)$. 

To prove the sufficiency, assume that elements $x\in X$ and $y\in Y$ satisfy $\lVert y\rVert<\lVert x\rVert$ and $(x_n)$ is a weakly null sequence in $X$. 

 For $\varepsilon>0$ such that $(1+\varepsilon)\lVert y\rVert<\lVert x\rVert$, we have
\begin{align*}
\limsup_n\lVert y+Tx_n\rVert&=\limsup_n\left\lVert\frac{1}{1+\varepsilon}[(1+\varepsilon)y+Tx_n]+\left(1-\frac{1}{1+\varepsilon}\right)Tx_n\right\rVert\\
&\leq\frac{1}{1+\varepsilon}\limsup_n\lVert(1+\varepsilon)y+Tx_n\rVert+\left(1-\frac{1}{1+\varepsilon}\right)\limsup_n\lVert Tx_n\rVert,\\
&<\frac{1}{1+\varepsilon}\limsup_n\lVert x+x_n\rVert+\left(1-\frac{1}{1+\varepsilon}\right)\limsup_n\lVert x+x_n\rVert\\
&=\limsup_n\lVert x+x_n\rVert
\end{align*}
which finishes the proof.
\end{proof}

We shall list a few basic observation on properties $(M)$ and $(O)$ which are useful to find a pair of Banach spaces having the WMP.

\begin{prop}\label{basicpro}For Banach spaces $X$ and $Y$,
\begin{enumerate}
\item $(X,X)$ has property $\mathcal{A}$  if and only if $X$ has property $\mathcal{A}$ where $\mathcal{A}$  is one of $(O)$, $(M)$ and strict $(M)$.
\item if $X$ and $Y$ have property $(M)$, then $(X,Y)$ has property $(M)$.
\item if $X$ has property $(O)$, then $(X,Y)$ has property $(O)$.
\item if $Y$ has property $(O)$ and $X$ has property $(M)$, then $(X,Y)$ has property $(O)$.
\item if $Y$ has property $(O)$ and $(X,Y)$ has property $(M)$, then $(X,Y)$ has property $(O)$.
\end{enumerate}
\end{prop}
\begin{proof}
(1) Note that the necessity of each statements is straightforward by their definition of properties $(O)$, $(M)$ and strict $(M)$. The statement of the sufficiency for $\mathcal{A}=(O)$ is obvious, and that for $\mathcal{A}=(M)$ is proved in \cite[Lemma 2.2]{K1}.

 To prove the one for  $\mathcal{A}=\text{strict~}(M)$,  fix elements $x,y\in X$ such that $\|y\|<\|x\|$, a contraction $T\in \mathcal{L}(X,Y)$ and a weakly null sequence $(x_n)\subset X$. From lemma \ref{sp}, we have that $id_X$ has property $(M)$ that is equivalent to $(X,X)$ has property $(M)$. Hence, we have 
$$\limsup_n\lVert y + id_X\circ Tx_n\rVert<\limsup_n\lVert x+  id_X\circ Tx_n\rVert\leq  \limsup_n\lVert x+x_n\rVert$$
which shows that $(X,X)$ has property strict $(M)$.

(2) This is observed in \cite[P. 171]{KW} without the proof, and it is followed by the slight modification of the proof of \cite[Lemma 2.2]{K1}.

(3) For a non-zero element $x\in X$, a contraction $T\in \mathcal{L}(X,Y)$ and a weakly null sequence $(x_n)\subset X$, we have
$$\limsup_n\lVert Tx_n\rVert\leq\limsup_n\lVert x_n\rVert<\limsup_n\lVert x+x_n\rVert$$
which shows that $(X,Y)$ has property $(O)$.

(4) Take a non-zero element $x\in X$, a contraction $T\in \mathcal{L}(X,Y)$ and a weakly null sequence $(x_n)\subset X$. Without loss of generality, we assume that $T\neq 0$, and  take $z\in X$ such that $\|Tz\|\neq 0$ and $\|z\|=\|x\|$. Then, we see that
$$\limsup_n\lVert Tx_n\rVert<\limsup_n\lVert Tz+Tx_n\rVert\leq\limsup_n\lVert z+x_n\rVert= \limsup_n\lVert x+x_n\rVert$$
which shows that $(X,Y)$ has property $(O)$.

(5) Take a non-zero element $x\in X$, a contraction $T\in \mathcal{L}(X,Y)$ and a weakly null sequence $(x_n)\subset X$. For a non-zero element $y\in Y$ such that $\|y\|\leq \|x\|$, we see that
$$\limsup_n\lVert Tx_n\rVert<\limsup_n\lVert y+Tx_n\rVert\leq \limsup_n\lVert x+x_n\rVert$$
which shows that $(X,Y)$ has property $(O)$.
\end{proof}

It is obvious that $\ell_p$ for $1\leq  p<\infty$ has both properties $(M)$ and $(O)$ and that $c_0$ has property $(M)$. Hence, proposition \ref{basicpro} shows that the following examples have the WMP.
\begin{example}\label{firstexample} For $1<p<\infty$ and $1\leq q<\infty$, pairs $(\ell_p,\ell_q)$ and $(\ell_p,c_0)$ have the WMP.
\end{example}

It is worth to remark that properties $(M)$ and $(O)$ for $X$ are hereditary properties. Indeed, if a Banach space $X$ has property $(M)$ or $(O)$, then so does every subspace $Y$ of $X$. This gives us the aforementioned example, a pair $(Z,Z)$ has the WMP but $\mathcal{K}(Z)$ is not an M-ideal in $\mathcal{L}(Z)$ by the famous Enflo-Davie subspace $Z$ of $\ell_p$ for $1<p<\infty$ which does not have the compact approximation property  (in short CAP) (see \cite{LT}). Hence, we see that an M-ideal of compact operators and the WMP are not related in general. However, if $Y$ is a Banach space such that $\mathcal{K}(\ell_p,Y)$ is an M-ideal in $\mathcal{L}(\ell_p,Y)$, then $(\ell_p,Y)$ has the WMP. It is followed by proposition \ref{basicpro} and the fact that, for $1<p<\infty$, $\mathcal{K}(\ell_p,Y)$ is an M-ideal in $\mathcal{L}(\ell_p,Y)$ if and only if $(\ell_p,Y)$ has property $(M)$ (see \cite{KW}). Actually, N. J. Kalton and D. Werner had dealt with more general setting in \cite{KW}, but we only state this for the convenience.

\begin{cor}\label{MimplyWMP}For a Banach space $Y$ and $1<p<\infty$, if $\mathcal{K}(\ell_p,Y)$ is an M-ideal in $\mathcal{L}(\ell_p,Y)$, then $(\ell_p,Y)$ has the WMP.
\end{cor} 

In fact, pairs $(\ell_p,\ell_q)$ and $(\ell_p,c_0)$ in example \ref{firstexample} are previously known ones having the WMP (see \cite{AGPT,GP}). In the following, we have new ones, pairs $(\ell_p,Y)$ given in \cite{KW} such that $\mathcal{K}(\ell_p,Y)$ is an M-ideal in $\mathcal{L}(\ell_p,Y)$. 

\begin{example}\label{secondexample}For $2\leq p<\infty$ and $1<q\leq 2$, the following pairs have the WMP.
\begin{enumerate}
\item $(\ell_p,L_p[0,1])$.
\item $(\ell_p,c_p[0,1])$ where $c_p$ is the Schatten class.
\item $(\ell_q,UT_q)$ where $UT_q$ is the subspace of $c_q$ consisting of operators whose matrix representations with respect to the unit vector basis in $\ell_2$ are upper trigonal.
\end{enumerate}
\end{example}

K. John and D. Werner showed in \cite{JW} that for a subspace $X$ of a quotient of $L_p[0,1]$ or $c_p$ ($2 \leq p < \infty$) having a $1$-unconditional Schauder decomposition or merely the unconditional MCAP, $(X, \ell_p)$ has property $(M)$. Hence, we also have the following example from proposition \ref{basicpro}.

\begin{example}\label{thirdexample}For $2\leq p<\infty$, if  a subspace $X$ of a quotient of $L_p[0,1]$ or $c_p$ ($2 \leq p < \infty$) having a $1$-unconditional Schauder decomposition or merely the unconditional MCAP, then $(X,\ell_p)$ has the WMP.
\end{example}

Very recently, L. C. Garc\'{\i}a-Lirola and C. Petitjean studied the WMP with moduli of asymptotic uniform convexity and smoothness (see \cite{GP}). The \textbf{modulus of asymptotic uniform convexity} of $X$ is a function $\bar{\delta}_X$ on $[0,\infty)$ given by 
$$\bar{\delta}_X(t) = \inf_{x\in S_X} \bar{\delta}_X(x,t) \text{~where~} \bar{\delta}_X(x,t)=\sup_{\dim(X/Y)<\infty}\inf_{y\in Y,\|y\|\geq t}\lVert x+y\rVert-1.$$
The \textbf{modulus of asymptotic uniform smoothness} of $X$ is a function $\bar{\rho}_X$ on $[0,\infty)$ given by
$$\bar{\rho}_X(t)=\sup_{x\in S_X} \bar{\rho}_X(x,t) \text{~where~} \bar{\rho}_X(x,t)=\inf_{\dim(X/Y)<\infty}\sup_{y\in Y,\|y\|\leq t}\lVert x+y\rVert-1.$$

If $X$ is a dual space, we can consider different moduli by taking all finite-codimensional weak* closed subspaces $Y$ of $X$, and these moduli are denoted by $\bar{\delta}^*_X$ and $\bar{\rho}^*_X$ respectively. A Banach space $X$ is said to be asymptotically uniform convex if $\bar{\delta}_X(t)>0$ for every $t>0$ and $X$ is said to be asymptotically uniform smooth if $\frac{\bar{\rho}_X(t)}{t} \to 0$ as $t\to0$. There has been intensive study on these moduli and we refer the reader to \cite{JLDS} for more  information. We note a few facts which can be found in this reference that $\bar{\delta}_X(t)$ and $\bar{\rho}_X(t)$ are non-decreasing Lipschitz functions with Lipschitz constant at most $1$, and the inequality $\bar{\delta}_X\leq \bar{\rho}_X$ holds. We also present the following useful inequalities. We omit their proofs since they must be known (see for example \cite[Proposition 2.2]{GP}).

\begin{rem}\label{easyinequality}
For a Banach space $X$, every bounded weakly null net $(x_\alpha)$ in $X$ and non-zero element $x$ in $X$ satisfy the following inequalities.
\begin{itemize}
\item [(1)] $\liminf_\alpha\lVert x+x_\alpha\rVert \geq \lVert x\rVert+\lVert x\rVert\bar{\delta}_X\left(\frac{x}{\|x\|},\frac{\liminf_\alpha\lVert x_\alpha\rVert}{\lVert x\rVert}\right)$.
\item [(2)] $\limsup_\alpha\lVert x+x_\alpha\rVert \leq \lVert x\rVert+\lVert x\rVert\bar{\rho}_X\left(\frac{x}{\|x\|},\frac{\limsup_\alpha\lVert x_\alpha\rVert}{\lVert x\rVert}\right)$.
\end{itemize}
In particular, if $\lim_\alpha \|x_\alpha\|$ exists, then 
\begin{align*}
\|x\|+\|x\|\bar{\delta}_X\left(\frac{\lim_\alpha \|x_\alpha\|}{\|x\|}\right)\leq \liminf_\alpha\lVert x+x_\alpha\rVert \leq\limsup_\alpha\lVert x+x_\alpha\rVert \leq \|x\|+\|x\|\bar{\rho}_X\left(\frac{\lim_\alpha \|x_\alpha\|}{\|x\|}\right).
\end{align*}
\end{rem}

With properties $(M)$ and $(O)$, we analyze \cite[Theorem 3.1]{GP} which shows that if a reflexive Banach space $X$ and a Banach spaces $Y$ satisfy $\bar{\delta}_X\geq \bar{\rho}_Y$ and $\bar{\delta}_X(t)>t-1$  for every $t\geq 1$, then $(X,Y)$ has the WMP.

\begin{theorem}\label{GPanalyze} For Banach spaces $X$ and $Y$,
\begin{enumerate}
\item if $(X,Y)$ satisfies $\bar{\delta}_X\geq \bar{\rho}_Y$, then it has property $(M)$.
\item if $X$ satisfies $\bar{\delta}_X(x,t)>t-1$ for every $t\geq 1$ and $x\in S_X$, then it has property $(O)$.
\end{enumerate}
In particular, if $(X,Y)$ satisfies $\bar{\delta}_X\geq \bar{\rho}_Y$ and $\bar{\delta}_X(x,t)>t-1$ for every $t\geq 1$ and $x\in S_X$, then it has the WMP.
\end{theorem} 
\begin{proof}
(1) Let $T\in \mathcal{L}(X,Y)$ be a contraction, $x\in X$ and $y\in Y$ satisfy $\| y\|\leq \| x\|$ and $(x_n)$ be a weakly null sequence in $X$. 

If $\|x\|=\|y\|=0$, then we have  $$\limsup_n\lVert y+Tx_n\rVert =\limsup_n\lVert Tx_n\rVert\leq  \limsup_n\lVert x_n\rVert = \limsup_n\lVert x+x_n\rVert.$$

Assume $\|x\|=\|y\|\neq 0$, and take a subsequence $x_{n_i}$ of $x_n$ such that $\limsup_n \| y+Tx_n\|=\lim_i \| y+Tx_{n_i}\|$ and $\limsup_i \|Tx_{n_i}\|=\lim_i \|Tx_{n_i}\|$. From remark \ref{easyinequality} we have that 
\begin{align*}
\limsup_n \| y+Tx_n\|
=&\lim_i\lVert y+Tx_{n_i}\rVert \\
\leq& \|y\|+\|y\|\bar{\rho}_Y\left(\frac{\lim_i\lVert Tx_{n_i}\rVert}{\lVert y\rVert}\right)\\
\leq& \|x\|+\|x\|\bar{\delta}_X\left(\frac{\lim_i\lVert Tx_{n_i}\rVert}{\lVert x\rVert}\right)\\
\leq&\|x\|+\|x\|\bar{\delta}_X\left(\frac{\liminf_i\lVert x_{n_i}\rVert}{\lVert x\rVert}\right)\\
\leq & \liminf_i\lVert x+x_{n_i}\rVert \leq \limsup_n \| x+x_n\|.
\end{align*}

Assume $\|x\|>\|y\|$, and take $v\in Y$ such that $\|v\|=\|x\|$ and $y=\lambda v + (1-\lambda) (-v)$ for some $0\leq \lambda \leq 1$. From the above assertion, we have 
$$
\limsup_n \| y+Tx_n\|\leq \limsup_n\max\left\{\|v +Tx_n\|,\|-v + Tx_n\|\right\}\leq \limsup_n \| x+x_n\|.
$$

(2) Since it is non-negative, we assume that $\bar{\delta}_X(u,t)>t-1$ for every $t\geq 0$ and $u\in S_X$. Let $x\in X$ be a non-zero element and  $(x_n)$ be a weakly null sequence in $X$. For a subsequence $x_{n_i}$ of $x_n$ such that $\limsup_n \|x_n\|=\lim_i \|x_{n_i}\|$, we have
\begin{align*}
\frac{\limsup_n\lVert x_{n}\rVert}{\lVert x\rVert}
&=\frac{\lim_i\lVert x_{n_i}\rVert}{\lVert x\rVert}\\
&<\bar{\delta}_X\left(\frac{x}{\lVert x\rVert},\frac{\lim_i\lVert x_{n_i}\rVert}{\lVert x\rVert}\right)+1\\
&\leq\frac{1}{\lVert x\rVert}\liminf_i\lVert x+x_{n_i}\rVert\\
&\leq\frac{1}{\lVert x\rVert}\limsup_n\lVert x+x_n\rVert
\end{align*}
which finishes the proof.
\end{proof}

In \cite{K1}, N. J. Kalton showed that whenever a separable Banach space $X$ satisfies that $\mathcal{K}(X)$ is an M-ideal in $\mathcal{L}(X)$, $X$ has property $(M_p)$ for some $1<p<\infty$ if and only if there exists $1 < p <\infty$ such that $X$ has \textbf{property $\bf (m_p)$} that means every weakly null sequence $(x_n)$ and every $x\in X$ satisfy
$$\limsup_{n} \| x + x_n\|^p=\|x\|^p + \limsup_{n}  \|x_n\|^p=\|(\|x\|,\limsup_n\|x_n\|)\|_p^p.$$ In order to get this, he showed that if an infinite dimensional separable Banach space $X$ has property $(M)$ and contains no isomorphic copy of $\ell_1$, then there exists $1<p\leq \infty$ and a weakly null sequence $(x_n)\subset S_X$ such that for every element $x\in X$ and every scalar $\alpha\in\mathbb{K}$, $$\lim_{n}\|x+\alpha x_n\|=\|(\|x\|,|\alpha|)\|_p.$$

The characterization of N. J. Kalton is generalized later by E. Oja \cite{Oja1} for arbitrary Banach spaces $X$, and it is true that for $1<p\leq \infty$ a Banach space $X$ has property  $(M_p)$ if and only if it has property $m_p$ and $\mathcal{K}(X)$ is an M-ideal in $\mathcal{L}(X)$ (see theorem \ref{Ojaminf} for the definition of  property $(m_\infty)$).

As a corollary of theorem \ref{GPanalyze} and the result of N. J. Kalton, we show that every reflexive space $X$ such that $\bar{\delta}_X=\bar{\rho}_X$ has property $(O)$.

\begin{cor}\label{asyp}
Let a Banach space $X$ contain no isomorphic copy of $\ell_1$. If $\bar{\delta}_X=\bar{\rho}_X$, then $X$ has property $(m_p)$ for some $1<p\leq\infty$. In particular, every reflexive space $X$ such that $\bar{\delta}_X=\bar{\rho}_X$ has property $(O)$.
\end{cor}
\begin{proof}
Since every finite dimensional space has property $(m_p)$ for every $1\leq p \leq \infty$ and property $(O)$, it is enough to assume that $X$ is infinite dimensional.

 By (1) of theorem \ref{GPanalyze}, we see that $X$ has property $(M)$, and so does an infinite dimensional separable subspace $Y$ of $X$. Hence, there exists $1<p\leq \infty$  and a weakly null sequence $(y_n)\subset S_Y$ such that for every $y\in Y$ and every $\alpha\in\mathbb{K}$, $$\lim_{n}\|y+\alpha y_n\|=\|(\|y\|,|\alpha|)\|_p$$ by \cite[Proposition 3.9]{K1}. From remark \ref{easyinequality}, we see that for every $z\in S_Y$ and every $t\geq 0$,
$$1+\bar{\delta}_X(t)\leq \lim_n\| z + t y_n\|\leq 1+\bar{\rho}_X(t)$$
which gives  $\bar{\rho}_X(t)=\bar{\delta}_X(t)=\|(1,t)\|_p-1$. We see that $X$ has property $(m_p)$ as a  consequence of remark \ref{easyinequality}. Indeed, for a weakly null sequence $(x_n)$, take a subsequence $(z_n)$ of $(x_n)$ such that $\lim_n\|z_n\|=\limsup_n\|x_n\|$. From  remark \ref{easyinequality} we have that 
\begin{align*}
\lVert x\rVert+\lVert x\rVert\bar{\delta}_X\left(\frac{x}{\|x\|},\frac{\lim_n\lVert z_n\rVert}{\lVert x\rVert}\right)
&\leq \liminf_n\lVert x+z_n\rVert\\
&\leq \limsup_n\lVert x+z_n\rVert\\
&\leq \limsup_n\lVert x+x_n\rVert\\
&\leq \lVert x\rVert+\lVert x\rVert\bar{\rho}_X\left(\frac{x}{\|x\|},\frac{\limsup_n\lVert x_n\rVert}{\lVert x\rVert}\right)
\end{align*}
for arbitrary non-zero element $x\in X$.
Hence, we see that 
$$\lVert x\rVert+\lVert x\rVert\bar{\delta}_X\left(\frac{\limsup_n\lVert x_n\rVert}{\lVert x\rVert}\right)
\leq \limsup_n\lVert x+x_n\rVert
 \leq \lVert x\rVert+\lVert x\rVert\bar{\rho}_X\left(\frac{\limsup_n\lVert x_n\rVert}{\lVert x\rVert}\right).$$
and
$$\limsup_n\lVert x+x_n\rVert=\|(\|x\|,\limsup_n\|x_n\|)\|_p.$$

If $X$ is an infinite dimensional reflexive space, we see that $X$ does not contain a subspace isomorphic to $c_0$. Hence, according to the proof of \cite[Proposition 3.9]{K1}, we see that $X$ has $m_p$ for some $1<p<\infty$. It is clear that $X$ has property $(O)$.
\end{proof}

We also present a list of known Banach spaces having property $(O)$.
 
\begin{example} The following Banach space $X$ has property $(O)$.

\begin{enumerate}
\item A uniformly convex Banach space $X$ having property $(M)$
\item A uniformly convex Banach space $X$ having a property that every element $x\in X$ and weakly null sequence $(x_n)\subset X$ satisfy 
$$\lim_n\left(\lVert x+x_n\rVert-\lVert -x+x_n\rVert\right)=0.$$
\item A uniformly convex Banach space $X$ having a sequence $(K_n)\subset \mathcal{K}(X)$ such that $K_n$ and $K^*_n$ converge strongly to $id_X$ and $id_{X^*}$ respectively and that
$$\lim_n\lVert id-2K_n\rVert = 1.$$
\item A Banach space $X$ defined by 
$$X=\left[\bigoplus_{i=1}^\infty X_i\right]_p$$
for Banach spaces $X_i$ having property $(O)$ and $1<p<\infty$.
\end{enumerate}
\end{example}
\begin{proof}
(1) This is proved in \cite[Theorem 2.3]{Xu} for more general setting.

(2) This is proved in \cite[Proposition 1 and 2]{Sim}.

(3) In  \cite[Theorem 2.1]{JW}, it is proved that $X$ satisfies the condition in (2).

(4) This is proved in \cite[Proposition 3.11]{Ha}.
\end{proof}

As a consequence of proposition \ref{basicpro}, if $X$ is reflexive, $(X,Y)$ has property $(M)$ and one of $X$ and $Y$ has property $(O)$, then $(X,Y)$ has the WMP. Hence, it is worth to present pairs $(X,Y)$ have property $(M)$ like (1) of theorem  \ref{GPanalyze}. For example, (1) of theorem \ref{GPanalyze} shows that if $\bar{\delta}_X= \bar{\rho}_X$ then $X$ has property  $(M)$. Hence, we see that if $X$ and $Y$ satisfy $\bar{\delta}_X= \bar{\rho}_X$ and $\bar{\delta}_Y= \bar{\rho}_Y$, then $(X,Y)$ has property $(M)$ by (2) of proposition \ref{basicpro}. A pair $(X,c_0)$ also has property $(M)$ for any Banach space $X$. Indeed, for any $y\in c_0$ and $x\in X$ such that $\|y\|\leq \|x\|$, a contraction $T\in \mathcal{L}(X,c_0)$ and a weakly null sequence $(x_n)\subset X$, we see that $$\|y\|\leq  \limsup_n\|x+x_n\| \text{~and~} \limsup_n\|Tx_n\|\leq \limsup_n\|T(x+x_n)\| \leq \limsup_n\|x+x_n\|.$$
Hence, we have $$\limsup_n\|y+Tx_n\|=\left\|\left(\|y\|,\limsup_n\|Tx_n\|\right)\right\|_\infty \leq \limsup_n\|x+x_n\|.$$
We summarize these as follows.

\begin{cor}\label{thirdexample} For Banach spaces $X$ and $Y$,
\begin{enumerate}
\item if $\bar{\delta}_X\geq\bar{\rho}_Y$,  $X$ is reflexive and one of $X$ and $Y$ has property  $(O)$, then $(X,Y)$ has the WMP.

\item if $\bar{\delta}_X=\bar{\rho}_X$, $\bar{\delta}_Y=\bar{\rho}_Y$ and $X$ is reflexive, then $(X,Y)$ has the WMP.

\item if $X$ is reflexive and has property $(O)$, then $(X,c_0)$ has the WMP.
\end{enumerate}
\end{cor}

From the result of D. van Dulst \cite{Dul} which shows  that for every reflexive Banach space $X$ there is a renorming $\tilde{X}$ of $X$ which has property $(O)$, we have the following.

\begin{cor}
For every reflexive Banach space $X$, there is a renorming $\tilde{X}$ of $X$ such that  $(\tilde{X},c_0)$ has the WMP.
\end{cor}

\section{Property $(sM^*)$ and the weak* to weak* maximizing property}\label{section4}

In his paper \cite{K1}, N.J. Kalton introduced property $(M^*)$ to characterize a separable Banach space $X$ such that $\mathcal{K}(X)$ is an M-ideal in $\mathcal{L}(X)$. Later, property $(sM^*)$, the net-version of property $(M^*)$, is introduced, and a characterization of Banach space $X$ such that $\mathcal{K}(X)$ is an M-ideal in $\mathcal{L}(X)$ is obtained \cite{HWW}. In this section, we consider property strict $(sM^*)$ to study the  weak*-to-weak*MP  that is introduced by L. C. Garc\'ia‑Lirola and C. Petitjean (see \cite{GP}). As is the case with property $(M)$ in section 3, we give the definitions of properties $(sM^*)$, strict $(sM^*)$ and $(sO^*)$ for a contraction and a pair of Banach spaces.\\

\begin{definition}~

\begin{enumerate}
\item For Banach spaces $X$ and $Y$, we say that  a contraction $T\in \mathcal{L}(X,Y)$ has \textbf{property $\bf (sM^*)$} if elements $x^*\in X^*$ and $y^*\in Y^*$ satisfy $\lVert y^*\rVert \geq \lVert x^*\rVert$, then $\limsup_\alpha\lVert x^*+T^*y^*_\alpha\rVert \leq \limsup_\alpha\lVert y^*+y^*_\alpha\rVert$
for every bounded weakly* null net $(y^*_\alpha) \subset Y$. 

A pair $(X,Y)$ is said to have property $(sM^*)$ if every contraction $T\in\mathcal{L}(X,Y)$ has property $(sM^*)$, and $X$ is said to have property $(sM^*)$ if the identity operator $id_X$ has property $(sM^*)$.

\item For Banach spaces $X$ and $Y$, we say that  a contraction $T\in \mathcal{L}(X,Y)$ has \textbf{property strict $\bf (sM^*)$} if elements $x^*\in X^*$ and $y^*\in Y^*$ satisfy $\lVert y^*\rVert > \lVert x^*\rVert$, then
 $\limsup_\alpha\lVert x^*+T^*y^*_\alpha\rVert < \limsup_\alpha\lVert y^*+y^*_\alpha\rVert$
for every bounded weakly* null net $(y^*_\alpha) \subset Y$. 

A pair $(X,Y)$ is said to have property strict $(sM^*)$ if every contraction $T\in\mathcal{L}(X,Y)$ has property strict $(sM^*)$, and $X$ is said to have property strict $(sM^*)$ if the identity operator $id_X$ has property strict $(sM^*)$.

\item For Banach spaces $X$ and $Y$, we say that  a contraction $T\in \mathcal{L}(X,Y)$ has \textbf{property $\bf (sO^*)$} if an element $y^*\in Y^*$ is non-zero, then
$\limsup_\alpha\lVert T^*y^*_\alpha\rVert<\limsup_\alpha\lVert y^*+y^*_\alpha\rVert$
for every bounded weakly* null net $(y^*_\alpha) \subset Y^*$. 

A pair $(X,Y)$ is said to have property $(sO^*)$ if every contraction $T\in\mathcal{L}(X,Y)$ has property  $(sO^*)$, and $X$ is said to have property $(sO^*)$ if the identity operator $id_X$ has property  $(sO^*)$.\\
\end{enumerate}
\end{definition}

Note that \textbf{properties $\bf (M^*)$, strict $\bf (M^*)$} and $\bf (O^*)$, the sequence versions of properties $(sM^*)$, strict $(sM^*)$ and $(sO^*)$, can be defined analogously, and we will use them without their formal definition. From net versions of their proofs, we can deduce many observations in section \ref{section3}, and we list a few of them without their proofs as follows. Note that the sequence version of each observation also holds.

\begin{prop}\label{propnet}
For Banach spaces $X$ and $Y$,
\begin{enumerate}
\item a contraction $T\in\mathcal{L}(X,Y)$ has property strict $(sM^*)$ if and only if it has both properties $(sM^*)$ and $(sO^*)$.
\item $(X,X)$ has property $\mathcal{A}$  if and only if $X$ has property $\mathcal{A}$ where $\mathcal{A}$  is one of $(sO^*)$, $(sM^*)$ and strict $(sM^*)$.
\item if $X$ and $Y$ have property $(sM^*)$, then $(X,Y)$ has property $(sM^*)$.
\item if $Y$ has property $(sO^*)$, then $(X,Y)$ has property $(sO^*)$.
\item if $X$ has property $(sO^*)$ and $Y$ has property $(sM^*)$, then $(X,Y)$ has property $(sO^*)$.
\item if $X$ has property $(sO^*)$ and $(X,Y)$ has property $(sM^*)$, then $(X,Y)$ has property $(sO^*)$.
\end{enumerate}
\end{prop}

We stress the weak*-to-weak*MP version of theorem \ref{MWMP} with strict $(sM^*)$. The proof is just a slight modification of the one of theorem \ref{MWMP}, but we give it for the completeness.

\begin{theorem}\label{netw*MP}
For Banach spaces $X$ and $Y$, if $(X,Y)$ has property strict $(sM^*)$, then $(Y^*,X^*)$ has the weak*-to-weak*MP.
\end{theorem}
\begin{proof}
For  an operator $T\in \mathcal{L}(X,Y)$, let $(y^*_n) \subset B_{Y^*}$ be a non-weakly* null maximizing sequence of $T^*$. Without loss of generality, we assume that $\|T\|=1$. Let $(y^*_\alpha)$ be a subnet of $(y^*_n)$ which converges weakly* to a non-zero element $y^*$ of $B_{Y^*}$ and define a weakly* null net $(w^*_\alpha)$  by $w^*_\alpha = y^*_\alpha -y^*$.

If $T^*$ does not attain its norm, then $\lVert T^*y^* \rVert < \lVert y^* \rVert$. Hence, we have
$$1=\limsup_{\alpha}{\lVert T^* y^*_\alpha \rVert} = \limsup_\alpha {\lVert T^*y^* + T^*w^*_\alpha\rVert}<\limsup_\alpha{\lVert y^* + w^*_\alpha \rVert }=\limsup_\alpha{\lVert y^*_\alpha \rVert} \leq 1$$
which leads to a contradiction.
\end{proof}

We do not know whether theorem \ref{netw*MP} holds if we replace property strict $(sM^*)$ by property strict $(M^*)$, the sequence version of property strict $(sM^*)$. The reason is that it cannot be said to exist a convergent subsequence of a given non-weakly* null maximizing sequence of a given adjoint operator. However, it is obvious that property strict $(M^*)$ for $(X,Y)$ implies the weak*-to-weak*MP in case that $B_{Y^*}$ is weak* sequentially compact.

\begin{theorem}\label{netw*MP111}
For Banach spaces $X$ and $Y$, if $(X,Y)$ has property strict $(M^*)$ and $B_{Y^*}$ is weak* sequentially compact, then $(Y^*,X^*)$ has the weak*-to-weak*MP.
\end{theorem}

As is the case with the WMP, moduli $\bar{\delta}^*_{X^*}$ and $\bar{\rho}^*_{X^*}$ are useful to study the weak*-to-weak*MP \cite[Theorem 3.1]{GP}. Since the analogue of remark \ref{easyinequality} is true for these moduli and a weakly* null net, we get the following by a slight modification of the proof of theorem \ref{GPanalyze}. We omit its proof.

\begin{theorem}\label{GPanalyzew*} For Banach spaces $X$ and $Y$,
\begin{enumerate}
\item if $(X,Y)$ satisfies $\bar{\delta}^*_{Y^*}\geq \bar{\rho}^*_{X^*}$ then it has property $(sM^*)$.
\item if $Y$ satisfies $\bar{\delta}^*_{Y^*}(y^*,t)>t-1$ for every $t\geq 1$ and $y^*\in S_{Y^*}$, then $Y^*$ has property $(sO^*)$.
\end{enumerate}
In particular, if $(X,Y)$ satisfies $\bar{\delta}^*_{Y^*}\geq \bar{\rho}^*_{X^*}$ and $\bar{\delta}^*_{Y^*}(y^*,t)>t-1$ for every $t\geq 1$ and $y^*\in S_{Y^*}$, then $(Y^*,X^*)$ has the weak*-to-weak*MP.
\end{theorem} 

For Banach spaces $X$ and $Y$, it is known that if $(Y^*,X^*)$ has the weak*-to-weak*MP, then $(X,Y)$ has the ACPP (see \cite[Remark 1.3]{GP}).

\begin{center}
\begin{tikzcd}
(Y^*,X^*)\text{~has the~}\text{weak}^*\text{-}\text{weak}^*\text{MP} \arrow[d, Rightarrow]    \\
                                       (X,Y)\text{~has the ACPP}\\
\mathcal{K}(X,Y)\text{~is an M-ideal} \arrow[u, Rightarrow] 
\end{tikzcd}
\end{center}

As is the case with the WMP, the ACPP and an M-ideal of compact operators, it is natural to look at the relationship between the weak*-to-weak*MP and the others. However, we only know a few. It is mentioned in \cite{GP} that if $X$ is non-reflexive and $Y$ is 1-dimensional, then the pair $(Y^*,X^*)$ has the weak*-to-weak*MP, but $(X,Y)$ does have WMP. On the other side, it is not known whether or not $(Y^*,X^*)$ has the weak*-to-weak*MP if $(X,Y)$ has the WMP. It is known that if $X$ is a separable reflexive Banach space having property $(m_p)$ for some $1<p<\infty$, then $X^*$ has property $(m_q)$ for the conjugate $q$ of $p$ (see \cite[Theorem 2.6]{KW}). Moreover, we see that $(X,X)$ has property strict $(M^*)$ from the sequence version of proposition \ref{propnet}. Hence, for the famous Enflo-Davie subspace $Z$ of $\ell_p$ for $1<p<\infty$ which does not have the CAP, $\mathcal{K}(Z,Z)$ is not an M-ideal in $\mathcal{L}(Z,Z)$ but $(Z^*,Z^*)$ has the  weak*-to-weak*MP by theorem \ref{netw*MP111}. We do not know the other implications, and we raise them as questions.\\

\begin{question}For Banach spaces $X$ and $Y$,
\begin{enumerate}
\item does $(Y^*,X^*)$ have the weak*-to-weak*MP if $(X,Y)$ has the WMP?
\item does $(Y^*,X^*)$ have the weak*-to-weak*MP if $(X,Y)$ has the ACPP?
\item does $(Y^*,X^*)$ have the weak*-to-weak*MP if $\mathcal{K}(X,Y)$ is an M-ideal?
\end{enumerate}
\end{question}

If the answer to question (2) is true, then the others are true either as we have seen in section \ref{section1}. We remark that question (1) is given in \cite[Question 5.12]{GP}, and we give a partial answer when $(Y^*,X^*)$ has property $(sO^*)$.

\begin{lemma}\label{dualWMPlem}
For Banach spaces $X$ and $Y$, assume that $(X,Y)$ has property $(sO^*)$. For a non-zero contraction $T\in\mathcal{L}(X,Y)$ and a maximizing sequence  $(y^*_n)\subset B_{Y^*}$ of $T^*$, if there is a weakly null sequence $(x_n)\subset B_X$ such that $\|T\|=\limsup_n|y_n^*Tx_n|$, then $(y^*_n)$ is weakly$^*$ null. In particular, whenever $Y$ is reflexive, the same happens if $(X,Y)$ has $(O^*)$.
\end{lemma}
\begin{proof}
Assume $(y^*_n)$ is not a weakly$^*$ null sequence. Since there exists a subnet $(y^*_{\alpha})$ of $(y^*_n)$ which converges weakly$^*$ to a non-zero element $y^*\in Y^*$, we have
\begin{align*}
\|T\|=&\limsup_{\alpha}\left<T^*y_{\alpha}^*,x_{\alpha}\right>\\
=&\limsup_{\alpha}\left<T^*y_{\alpha}^*-T^*y^*,x_{\alpha}\right>\\
\leq&\|T\|\limsup_{\alpha}\left\|\frac{T^*}{\|T\|}(y_{\alpha}^*-y^*)\right\|\limsup_{\alpha}\|x_{\alpha}\|\\
<&\|T\|\limsup_{\alpha}\|y^*+(y_{\alpha}^*-y^*)\|\leq \|T\|
\end{align*}
which leads to a contradiction.

If $Y$ is reflexive and $(X,Y)$ has $(O^*)$, then the same holds since we can take a subsequence of $(y^*_n)$ instead of a subnet.
\end{proof}

\begin{theorem}\label{WMPWWMP}
 For Banach spaces $X$ and $Y$, if a pair $(X,Y)$ has property $(sO^*)$ and  the WMP, then $(Y^*,X^*)$ has the weak*-to-weak*MP. In particular, whenever $Y$ is reflexive, the same happens if $(X,Y)$ has $(O^*)$.
\end{theorem}
\begin{proof}
It is enough to prove that for an operator $T\in \mathcal{L}(X,Y)$ whose adjoint $T^*$ does not attains its norm, every maximizing sequence $(y_n^*)\subset B_{Y^*}$ of $T^*$ is weakly* null.

For a sequence $(x_n)\subset S_X$ such that $\|T^*y^*_n\|-\frac{1}{n}\leq |y^*_nTx_n|$, we see that $(x_n)$ is a maximizing sequence of $T$. Hence, it is weakly null. Otherwise, $T$ attains its norm since $(X,Y)$ has the WMP, and this implies $T^*$ attains its norm. Therefore, $(y^*_n)$ is weakly$^*$ null by lemma \ref{dualWMPlem}.
\end{proof}

By the sequence version of (4) in proposition \ref{propnet}, we see that $(X,\ell_p)$ has property $(O^*)$ for every $1<p<\infty$ and $X$. Hence, we have the following corollary.

\begin{cor}
For a Banach space $X$ and $1<p<\infty$, if $(X,\ell_p)$ has the WMP, then $(\ell_q,X^*)$ has the weak*-to-weak*MP where $q$ is the conjugate of $p$.
\end{cor}

From example \ref{thirdexample}, we have the following example.

\begin{example}
For $2\leq p<\infty$, if  a subspace $X$ of a quotient of $L_p[0,1]$ or $c_p$ ($2 \leq p < \infty$) having a $1$-unconditional Schauder decomposition or merely the unconditional MCAP, then $(\ell_q,X^*)$ has the weak*-to-weak*MP where $q$ is the conjugate of $p$.
\end{example}
We provide another example having the weak*-to-weak*MP with the direct sum of Banach spaces.
\begin{prop}\label{prop:stability}
For a Banach space $X$, a sequence of Banach spaces $(Y_i)$, and 
$$Y=\left[\bigoplus_{i=1}^\infty Y_i\right]_{c_0},$$
\begin{enumerate}
\item if $(X,Y_i)$ has the property $(sM^*)$ for each $i\in\mathbb{N}$, then $(X,Y)$ has the property $(sM^*)$.
\item if $(X,Y_i)$ has the property $(sO^*)$ for each $i\in\mathbb{N}$, then $(X,Y)$ has the property $(sO^*)$.
\end{enumerate}
\end{prop}
\begin{proof}
To prove (1), we first demonstrate the case $Y=Y_1\oplus_\infty Y_2$. Take elements $x^*\in X^*$ and $(\xi^*,\eta^*)\in Y^*$ with $\lVert x^*\rVert\leq \lVert(\xi^*,\eta^*)\rVert$, a bounded weakly$^*$ null net $(y^*_{\alpha,1},y^*_{\alpha,2})\subset Y^*$, and contraction $T\in \mathcal{L}(X,Y)$. It is clear that, both $(y^*_{\alpha,1})$ and $(y^*_{\alpha,2})$ are weakly$^*$ null and that $T^*=S^*_1\oplus S^*_2$ for some contraction $S_i \in\mathcal{L}(X,Y_i)$ $(i=1,2)$. 

Since it is enough to assume that  $(\xi^*,\eta^*)$ is non-zero, define $\lambda=\frac{\lVert \xi^* \rVert}{\|(\xi^*,\eta^*)\|}$. Then, we have 
$$\lambda\lVert x^*\rVert \leq \lVert \xi^* \rVert, \text{~and~}(1-\lambda)\lVert x^* \rVert \leq \|\eta^*\|.$$

 Let us take a subnet $(y^*_{\beta,1},y^*_{\beta,2})$ of $(y^*_{\alpha,1},y^*_{\alpha,2})$ such that 
$$\limsup_\alpha\|x^*+S^*_1y^*_{\alpha,1}+S^*_2y^*_{\alpha,2}\|=\lim_\beta\|x^*+S^*_1y^*_{\beta,1}+S^*_2y^*_{\beta,2}\|,$$ and limits $\lim_\beta\|(\xi^*,\eta^*)+(y^*_{\beta,1},y^*_{\beta,2})\|$, $\lim_\beta\|\lambda x^*+S_1^*y^*_{\beta,1}\|$, $\lim_\beta\|(1-\lambda) x^*+S_2^*y^*_{\beta,2}\|$, and $\lim_\beta\|\xi^*+y^*_{\beta,1}\|$ exist. Note that $\lim_\beta\|\eta^*+y^*_{\beta,2}\|$ also exists. Then, we have
\begin{align*}
\limsup_\alpha\lVert x^* + T^*(y^*_{\alpha,1},y^*_{\alpha,2})\rVert =&\lim_\beta\|x^*+S^*_1y^*_{\beta,1}+S^*_2y^*_{\beta,2}\|\\
\leq&\lim_\beta\lVert \lambda x^* + S^*_1y^*_{\beta,1}\rVert+\lim_\beta\lVert (1-\lambda)x^*+S^*_2y^*_{\beta,2}\rVert\\
\leq&\lim_\beta\lVert \xi^*+y^*_{\beta,1}\rVert+ \lim_\beta\lVert \eta^*+y^*_{\beta,2}\rVert\\
=& \lim_\beta\lVert (\xi^*,\eta^*)+(y^*_{\beta,1},y^*_{\beta,2})\rVert\\
\leq&\limsup_\alpha\|(\xi^*,\eta^*)+(y^*_{\alpha,1},y^*_{\alpha,2})\|
\end{align*}
Hence, $(X,Y)$ has property $(sM^*)$ for  $Y=Y_1\oplus_\infty Y_2$, and it is also true for $Y=\left[\bigoplus_{i=1}^n Y_i\right]_{\infty}$ for any $n\in \mathbb{N}$ by the induction.

Let $Y=\left[\bigoplus_{i=1}^\infty Y_i\right]_{c_0}$. Fix elements $y^*\in Y^*$ and $x^*\in X^*$ with $\lVert x^*\rVert\leq\lVert y^*\rVert$, a bounded weakly$^*$ null net $(y^*_\alpha)\subset Y^*$ and a contraction $T\in \mathcal{L}(X,Y)$. Without loss of generality, we assume $y^*\neq0$.

For $\varepsilon>0$, choose $n\in\mathbb{N}$ such that $\|(id_{Y^*}-P_n)y^*\|<\min\{\varepsilon,\|y^*\|-(1-\varepsilon)\|x^*\|\}$ where $P_n$ is the canonical projection from $Y^*=\left[\bigoplus_{i=1}^\infty Y_i^*\right]_1$ to the first $n$ coordinates. Note that $P_ny^*_\alpha$ is weakly$^*$ null and 
$$\|P_ny^*\|=\|y^*\|-\|(id_{Y^*}-P_n)y^*\|>(1-\varepsilon)\|x^*\|.$$

Take a subnet $(y^*_\beta)$ of $(y^*_\alpha)$ such that 
$$\lim_\beta\|(1-\varepsilon)x^*+T^*y^*_\beta\|=\limsup_\alpha\|(1-\varepsilon)x^*+T^*y^*_\alpha\|$$ and limits $\lim_\beta\| x^*+T^*P_ny_\beta^*\|$, $\lim_\beta\|T^*(id_{Y^*}-P_n)y_\beta^*\|$, $\lim_\beta\|P_n(y^*+y^*_\beta)\|$, $\lim_\beta\|(id_{Y^*}-P_n)y_\beta^*\|$, $\lim_\beta\|(id_{Y^*}-P_n)(y^*+y_\beta^*)\|$ and $\lim_\beta\| y+y_\beta\|$ exist. Since $\left(X,\left[\bigoplus_{i=1}^n Y_i\right]_{\infty}\right)$ has property $(sM^*)$, we have that
\begin{align*}
\limsup_\alpha\|(1-\varepsilon)x^*+T^*y^*_\alpha\|=&\lim_\beta\|(1-\varepsilon)x^*+T^*y^*_\beta\|\\
\leq&\lim_\beta\|(1-\varepsilon)x^*+T^*P_ny^*_\beta\|+\lim_\beta\|T^*(id_{Y^*}-P_n)y_\beta^*\|\\
\leq&\lim_\beta\|P_ny^*+P_ny_\beta^*\|+\left|\lim_\beta\|(id_{Y^*}-P_n)y_\beta^*\|-\|(id_{Y^*}-P_n)y^*\|\right|+\varepsilon\\
\leq&\lim_\beta\|P_n(y^*+y_\beta^*)\|+\lim_\beta\|(id_{Y^*}-P_n)(y^*+y^*_\beta)\|+\varepsilon\\
=&\lim_\beta\|y^*+y_\beta^*\|+\varepsilon\\
\leq&\limsup_\alpha\|y^*+y^*_\alpha\|+\varepsilon.
\end{align*}
Since $\varepsilon$ is arbitrary, we have $\limsup_\alpha\|x^*+T^*y^*_\alpha\|\leq\limsup_\alpha\|y^*+y^*_\alpha\|$. Hence, the pair $(X,Y)$ has property $(sM)$.

To prove (2), as in the proof of (1), we demonstrate the case $Y=Y_1\oplus_\infty Y_2$. Take a non-zero element $(\xi^*,\eta^*)\in Y^*$, a bounded weakly$^*$ null net $(y_{\alpha,1}^*,y_{\alpha,2}^*)\subset Y^*$ and a contraction $T\in\mathcal{L}(X,Y)$, and let $T^*=S^*_1\oplus S^*_2$ for some contractions $S_i\in\mathcal{L}(X,Y_i)$ $(i=1,2)$. 

Let us take a subnet $(y_{\beta,1}^*,y_{\beta,2}^*)$ of $(y_{\alpha,1}^*,y_{\alpha,2}^*)$ such that 
$$\limsup_\alpha\|T^*(y_{\alpha,1}^*,y_{\alpha,2}^*)\|=\lim_\beta\|T^*(y_{\beta,1}^*,y_{\beta,2}^*)\|,$$ 
and both limits $\lim_\beta\|\xi^*+y^*_{\beta,1}\|$ and $\lim_\beta\|\eta^*+y^*_{\beta,2}\|$ exist. Then, we have
\begin{align*}
\limsup_\alpha\|T^*(y_{\alpha,1}^*,y_{\alpha,2}^*)\|=&\lim_\beta\|T^*(y_{\beta,1}^*,y_{\beta,2}^*)\|\\
\leq&\limsup_\beta\|S_1^*y^*_{\beta,1}\|+\limsup_\beta\|S_2^*y^*_{\beta,2}\|\\
<&\lim_\beta\|\xi^*+y^*_{\beta,1}\|+\lim_\beta\|\eta^*+y^*_{\beta,2}\|\\
=&\lim_\beta\|(\xi^*,\eta^*)+(y^*_{\beta,1},y^*_{\beta,2})\|\\
\leq&\limsup_\alpha\|(\xi^*,\eta^*)+(y^*_{\alpha,1},y^*_{\alpha,2})\|.
\end{align*}

Hence, $(X,Y)$ has property $(sO^*)$ for  $Y=Y_1\oplus_\infty Y_2$, and it is also true for $Y=\left[\bigoplus_{i=1}^n Y_i\right]_{\infty}$ for any $n\in \mathbb{N}$ by the induction.

In order to prove $(X,Y)$ has property $(sO^*)$ for $Y=\left[\bigoplus_{i=1}^\infty Y_i\right]_{c_0}$, we see that $\limsup_\alpha\|T^*y^*_\alpha\|\leq\limsup_\alpha\|y^*+y^*_\alpha\|$ for any non-zero element $y^*\in Y^*$, any bounded weakly$^*$ null net $(y_\alpha^*)\subset Y^*$ and any contraction $T\in \mathcal{L}(X,Y)$.

For given $\varepsilon>0$,  choose $n\in\mathbb{N}$ such that $\|(id_{Y^*}-P_n)y^*\|<\min\{\varepsilon,\|y^*\|\}$ where $P_n$ is the canonical projection in the proof of (1).

  Take a subnet $(y_\beta^*)$ of $(y_\alpha^*)$ such that $\limsup_\alpha\|T^*y^*_\alpha\|=\lim_\beta\|T^*y^*_\beta\|$ and limits $\lim_\beta\|P_n(y^*+y^*_\beta)\|$, $\lim_\beta\|(id_{Y^*}-P_n)y_\beta^*\|$ and $\lim_\beta\|(id_{Y^*}-P_n)(y^*+y_\beta^*)\|$ exist. Since $\left(X,\left[\bigoplus_{i=1}^n Y_i\right]_{\infty}\right)$ has property $(sO^*)$, we have that
\begin{align*}
 \limsup_\alpha\|T^*y^*_\alpha\|=&\lim_\beta\|T^*y^*_\beta\|\\
 \leq&\limsup_\beta\|T^*P_ny^*_\beta\|+\limsup_\beta\|T^*(id_{Y^*}-P_n)y_\beta^*\|\\
 <&\lim_\beta\|P_ny^*+P_ny^*_\beta\|+\lim_\beta\|(id_{Y^*}-P_n)y_\beta^*\|\\
 \leq &\lim_\beta\|P_n(y^*+y^*_\beta)\|+\lim_\beta\|(id_{Y^*}-P_n)(y^*+y_\beta^*)\|+\varepsilon\\
 =&\lim_\beta\|y^*+y^*_\beta\|+\varepsilon\\
 \leq&\limsup_\alpha\|y^*+y^*_\alpha\|+\varepsilon.
\end{align*}
Since $\varepsilon$ is arbitrary, we have $\limsup_\alpha\|T^*y^*_\alpha\|\leq\limsup_\alpha\|y^*+y^*_\alpha\|$. In fact, this holds for arbitrary $y^*$ clearly, and so we can refine the previous inequalities. Indeed, we have
$$\limsup_\beta\|T^*(id_{Y^*}-P_n)y_\beta^*\|\leq \limsup_\beta\|(id_{Y^*}-P_n)y^*+(id_{Y^*}-P_n)y_\beta^*\|$$
since $(id_{Y^*}-P_n)y_\beta^*$ is weakly* null, and this shows that for $y^*\neq 0$
\begin{align*}
\limsup_\alpha\|T^*y^*_\alpha\|<\limsup_\alpha\|y^*+y^*_\alpha\|.
\end{align*}
\end{proof}

\begin{example}For arbitrary $p,q\in (1,\infty)$, a pair $(\ell_1(\ell_p),\ell_q)$ has the weak*-to-weak*MP.
\end{example}

\begin{remark} As we use nets to define $(sM^*)$, we can consider the net-versions of the WMP and the weak*-to-weak*MP analogously. Moreover, with the same proofs of theorem \ref{netw*MP}, we see that for Banach spaces $X$ and $Y$ if $(X,Y)$ has property strict $(sM^*)$, then a pair $(Y^*,X^*)$ has the net version of the weak*-to-weak*MP that means an adjoint $T^*$ of $T\in \mathcal{L}(X,Y)$ attains its norm whenever there is a non-weakly* null net $(y^*_\alpha)\subset B_{Y^*}$ such that $$\lim_\alpha\lVert T^*y^*_\alpha\rVert=\lVert T^*\rVert.$$
 However, there is no need to consider the net version of the weak*-to-weak*MP since the existence of a non-weakly* null maximizing net of $T^*$ is equivalent to the that of a non-weakly* null maximizing sequence of $T^*$ (see \cite[Remark 1.2]{GP}). Similarly, we see that if a pair $(X,Y)$ has property strict $(sM)$, the net version of property $(M)$, then it has the net version of the WMP. However, it is not needed to consider this property by the same reason.
\end{remark}

\section{Properties $(sM)$ and $(sO)$ and the moduli of
asymptotic uniform convexity and smoothness}\label{section5}

The use of the moduli of asymptotic uniform convexity and smoothness has the advantage of helping us investigate properties $(sM)$ and $(sO)$ for Banach spaces, the net versions of properties $(M)$ and $(O)$. As a result, we characterize a Banach space with property $(sO)$. We also study a Banach space $X$ such that $\bar{\delta}_X=\bar{\rho}_X$, and this allows us to revisit the characterization of a Banach space with property $(M_p)$ given in \cite{K1} and \cite{Oja1}. To this end, we provide definitions of properties $(sM)$ and $(sO)$ and some useful lemmas. We refer the reader to \cite[Ch.VI.4]{HWW} for more details on property $(sM)$.\\

\begin{definition}~
\begin{enumerate}
\item We say that a Banach space $X$ has property $(sM)$ if elements $x,y\in X$ satisfy $\lVert y\rVert \leq \lVert x\rVert$, then $\limsup_\alpha\lVert y+x_\alpha\rVert \leq \limsup_\alpha\lVert x+x_\alpha\rVert$ for every bounded weakly null net $(x_\alpha) \subset X$.

\item  We say that a Banach space $X$ has property $(sO)$ if an element $x\in X$ is non-zero, then $\limsup_\alpha\lVert x_\alpha\rVert < \limsup_\alpha\lVert x+x_\alpha\rVert$ for every bounded weakly null net $(x_\alpha) \subset X$.
\end{enumerate}
\end{definition}

\begin{lemma}\cite[Lemma 4.13, Ch.VI.4]{HWW}\label{Meq} For a Banach space $X$, the following are equivalent.
\begin{enumerate}[(i)]
\item $X$ has property $(sM)$ 
\item Whenever elements $x,y\in X$ satisfy $\|y\|=\| x\|$ and $(x_\alpha)\subset X$ is a bounded weakly null net, if $\lim_\alpha \lVert x+x_\alpha\rVert$ exists, then $\lim_\alpha \| y+x_\alpha\|$ exists and $\lim_\alpha \| y+x_\alpha\|=\lim_\alpha \| x+x_\alpha\|$.
\end{enumerate}
\end{lemma}
\begin{lemma}\label{AUM}
For an infinite dimensional Banach space $X$, an element $x\in S_X$ and a real number $t>0$, there exist weakly null nets $(x_\alpha) \subset S_X$ and $(y_\beta)\subset S_X$ such that 
$$\bar{\delta}_X(x,t)=\lim_\alpha\lVert x+tx_\alpha\rVert-1 \text{~and~} \bar{\rho}_X(x,t)=\lim_\beta\lVert x+ty_\beta\rVert-1.$$
\end{lemma}
\begin{proof}
We only prove that there exists a weakly null net $(x_\alpha) \subset S_X$ such that $\bar{\delta}_X(x,t)=\lim_\alpha\lVert x+tx_\alpha\rVert-1$ since the rest can be proved by the same way.

We first note that $\bar{\delta}_X(x,t)=\sup_{\dim(X/Y)<\infty}\inf_{y\in Y,y\in S_Y}\lVert x+ty\rVert-1$ which is obvious by the convexity of the norm.

 Define a set $\mathcal{F}=\{(F,\varepsilon) : F\text{~is a finite subset of~}X^*, \varepsilon>0\}$ and a partial order $\leq$ by $\alpha\leq \beta$ for $\alpha = (F_\alpha, \varepsilon_\alpha)\in \mathcal{F}$ and $\beta = (F_\beta, \varepsilon_\beta)\in \mathcal{F}$ if $F_\alpha\subset F_\beta$ and $\varepsilon_\alpha > \varepsilon_\beta$.  

For each $\alpha= (F_\alpha, \varepsilon_\alpha)\in \mathcal{F}$ and $Y_\alpha=\bigcap_{x^*\in F_\alpha} \ker(x^*)$, take an element $x_\alpha \in S_{Y_\alpha}$ such that $\inf_{y\in S_{Y_\alpha}}\lVert x+ty\rVert +\varepsilon_\alpha >\lVert x+tx_\alpha\rVert$. It is clear that $(x_\alpha)$ is weakly null and that $\liminf_\alpha\lVert x+tx_\alpha\rVert-\|x\| \leq \bar{\delta}_X(x,t)$. From remark \ref{easyinequality}, we finish the proof.
\end{proof}

We characterize an infinite dimensional Banach space $X$ satisfying the condition in theorem \ref{GPanalyze} that $\bar{\delta}_X(x,t)>t-1$ for every $t\geq 1$ and $x\in S_X$ with property $(sO)$.

\begin{prop}\label{aum1}
For a Banach space $X$, $X$ has property $(sO)$ if and only if it is finite dimensional or satisfies $\bar{\delta}_X(x,t)>t-1$ for every $t\geq 1$ and $x\in S_X$.
\end{prop}
\begin{proof} Since every finite dimensional space has property $(sO)$, it is enough to assume that $X$ is infinite dimensional. 

In order to prove the necessity, assume  $X$ has property $(sO)$ and fix $x\in S_X$ and $t\geq 1$. By lemma \ref{AUM}, take a normalized weakly null net $(x_\alpha)$ such that $\bar{\delta}_X(x,t)=\lim_\alpha\lVert x+tx_\alpha\rVert-1$. Hence, we have 
$$\bar{\delta}_X(x,t)=\lim_\alpha\lVert x+tx_\alpha\rVert-1>\limsup_\alpha\lVert tx_\alpha\rVert-1=t-1.$$

Since the sufficiency can be proved by a slight modification of the proof of (2) of theorem \ref{GPanalyze}, we omit it.
\end{proof}

We observed in theorem \ref{GPanalyze} that if a Banach space $X$ satisfies $\bar{\delta}_X= \bar{\rho}_X$, then it has property $(M)$. We show that the converse holds for infinite dimensional weak-null-net stable Banach spaces. 
\begin{definition}\label{netstabledef}
A Banach space $X$ is said to be \textbf{net-stable} (resp. \textbf{weak-net-stable}) if, for any pair of bounded (resp. weakly null) nets $(x_\alpha)$ and $(y_\beta)$ in $X$, it satisfies that
$$\lim_\alpha\lim_\beta\lVert x_\alpha+y_\beta\rVert = \lim_\alpha\lim_\beta\lVert x_\alpha+y_\beta\rVert,$$
whenever both sides exist. 
\end{definition}
This is motivated by stable Banach spaces introduced by J.L. Krivine and B. Maurey \cite{KM}. Note that a separable Banach space $X$ is said to be \textbf{stable} if it satisfies the condition in \ref{netstabledef} for bounded sequences instead of bounded nets.

\begin{lemma}\label{AUM1}
For an infinite dimensional Banach space $X$ with property $(sM)$ and a weakly null net $(x_\alpha)\subset S_X$,
\begin{enumerate}
\item if $\bar{\delta}_X(x,t)=\lim_\alpha \|x+tx_\alpha\|-1$  for some $x\in S_X$, then
$$\bar{\delta}_X(t)=\bar{\delta}_X(z,t)=\lim_\alpha \|z+tx_\alpha\|-1 \text{~for every~} z\in S_X.$$ 
\item if $\bar{\rho}_X(x,t)=\lim_\beta \|x+tx_\alpha\|-1$ for some $x\in S_X$, then
$$\bar{\rho}_X(t)=\bar{\rho}_X(z,t)=\lim_\beta \|z+tx_\alpha\|-1 \text{~for every~} z\in S_X.$$ 
\end{enumerate}
\end{lemma}
\begin{proof}
We only prove (1) since the rest can be shown by the same proof. 

Take arbitrary $z\in S_X$. From remark \ref{easyinequality}, lemma \ref{Meq} and property $(sM)$, we have 
$$\bar{\delta}_X(x,t)=\lim_\alpha \|x+tx_\alpha\|-1= \lim_\alpha \|z+tx_\alpha\|-1\geq \bar{\delta}_X(z,t).$$
From lemma \ref{AUM}, there exists a weakly null net $(z_\alpha)\subset S_X$ such that $\bar{\delta}_X(z,t)=\lim_\alpha \|z+tz_\alpha\|-1$. Hence, we see that  $\bar{\delta}_X(z,t)\geq \bar{\delta}_X(x,t)$ by the same argument above, and this shows $(1)$.
\end{proof}

\begin{lemma}\label{asymmetric}
For an infinite dimensional weak-net-stable Banach space $X$ with property $(sM)$, an element $x\in S_X$, a  weakly null net $(x_\alpha)\subset S_X$ and real numbers $a,b\in\mathbb{R}$, if $\lim_\alpha\lVert ax+bx_\alpha\rVert$ exists, then $\lim_\alpha\lVert bx+ax_\alpha\rVert$ exists and 
$$\lim_\alpha\lVert ax+bx_\alpha\rVert=\lim_\alpha\lVert bx+ax_\alpha\rVert.$$
\end{lemma}
\begin{proof}
We first show that $\lim_\alpha\lVert ax+bx_\alpha\rVert=\limsup_\alpha\lVert bx+ax_\alpha\rVert$. Indeed, take a subnet $(y_\beta)$ of $(x_\alpha)$  such that $\lim_\beta\lVert bx+ay_\beta\rVert=\limsup_\alpha\lVert bx+ax_\alpha\rVert$, and let $(y_\eta)$ be the same net with $(y_\beta)$.
From lemma \ref{Meq}, we have
\begin{align*}
\lim_\alpha\lVert ax+bx_\alpha\rVert
&=\lim_\beta\lVert ax+by_\beta\rVert=\lim_\eta\lim_\beta\lVert ay_\eta+by_\beta\rVert\\
&=\lim_\beta\lim_\eta\lVert ay_\eta+by_\beta\rVert=\lim_\eta\lVert bx+ay_\eta\rVert=\limsup_\alpha\lVert bx+ax_\alpha\rVert
\end{align*}
Similarly, we see $\lim_\alpha\lVert ax+bx_\alpha\rVert=\liminf_\alpha\lVert bx+ax_\alpha\rVert$. Hence, we finish the proof.
\end{proof}

\begin{theorem}\label{Msym}
If an infinite dimensional weak-net-stable Banach space $X$ has property $(sM)$, then $\bar{\delta}_X=\bar{\rho}_X$. In particular, every infinite dimensional net-stable Banach space $X$ with property $(sM)$ satisfies $\bar{\delta}_X=\bar{\rho}_X$.
\end{theorem}
\begin{proof}
 Fix arbitrary $t>0$ and $x\in S_X$, and take weakly null nets $(\eta_\alpha)\subset S_X$ and $(\xi_\beta)\subset S_X$ such that $\bar{\delta}_X(x,t)+1=\lim_\alpha\lVert x+t\eta_\alpha\rVert$ and $\bar{\rho}_X(x,t)+1=\lim_\beta\lVert x+t\xi_\beta\rVert$ by lemma \ref{AUM}. From lemmas \ref{Meq}, \ref{AUM1} and \ref{asymmetric}, we see that
\begin{align*}
\bar{\rho}_X(t)+1&=\bar{\rho}_X(x,t)+1\\
&=\lim_\beta\lVert x+t\xi_\beta\rVert=\lim_\alpha\lim_\beta\lVert \eta_\alpha+t\xi_\beta\rVert\\
&=\lim_\beta\lim_\alpha\lVert t\xi_\beta+\eta_\alpha\rVert=\lim_\alpha\lVert tx+\eta_\alpha\rVert \\
&=\lim_\alpha\lVert x+t\eta_\alpha\rVert=\bar{\delta}_X(x,t)+1=\bar{\delta}_X(t)+1.
\end{align*}
\end{proof}

The converse of theorem \ref{Msym} holds if a norm induced by $\bar{\rho}_X$ is symmetric. For an infinite dimensional Banach space $X$, $\overline{\rho}_X$ is a $1$-Lipschitz convex function such that $\overline{\rho}_X(t)/t$ converges to $1$ as $t$ goes to infinity. Hence, it  induces an norm $N^{A}_X(a,b)$ on $\mathbb{R}^2$ (see \cite{K2}).
 $$N^{A}_X(a,b)=\begin{cases}\lvert a\rvert+\lvert a\rvert\bar{\rho}_X\left(\left\lvert \frac{b}{a}\right\rvert\right) &\text{if }a\neq0\\ \lvert b\rvert &\text{if }a=0\end{cases}$$ 
 It is worth to note that if $\bar{\rho}_X(t)\geq t-1$ for every $t\geq 1$, then $N^A_X$ is a lattice norm that means $N_X^A(a,b)\leq N_X^A(c,d)$ whenever $a\leq c$ and $b\leq d$ (see \cite[Lemma 2.3]{GP}). If $X$ has property $(sM)$, then we have $\limsup_\alpha\lVert x_\alpha\rVert\leq\limsup_\alpha\lVert x+x_\alpha\rVert$  for any weakly null net $(x_\alpha)$ and any element $x\in X$. Therefore, lemma \ref{AUM} implies $\bar{\rho}_X(t)\geq t-1$ for every $t\geq 1$, and thus $N^A_X$ is a lattice norm.  We consider a specific condition for $N^A_X$, the symmetricity.

\begin{definition} For a norm $N$ on $\mathbb{R}^2$, we say that $N$ is \textbf{symmetric} if $N(a,b)=N(b,a)$ for every $(a,b)\in \mathbb{R}^2$.
\end{definition}
The typical example of Banach space $X$ such that $N^A_X$ is symmetric is $\ell_p$ ($1\leq p<\infty$) since $\bar{\rho}_{\ell_p}(t)=(1+t^p)^{1/p}-1$. We note that if $X$ is an infinite dimensional weak-net-stable Banach space with property $(sM)$, then $N^A_X$ is symmetric as a consequence of remark \ref{easyinequality}, lemmas \ref{asymmetric} and \ref{AUM} and theorem \ref{Msym}. 

\begin{theorem}\label{A}
For an infinite dimensional Banach space $X$ such that $N^A_X$ is symmetric, $X$ is weak-net-stable and has property $(sM)$ if and only if $\bar{\delta}_X=\bar{\rho}_X$.
\begin{proof}
From theorem \ref{Msym}, we only need to prove the sufficiency. Indeed, assume $\bar{\delta}_X=\bar{\rho}_X$. Since $X$ has $(sM)$ by theorem \ref{GPanalyze}, we prove $X$ is weak-net-stable . Take a pair of bounded weakly null nets $(\eta_\alpha)$ and $(\xi_\beta)$ such that both limits $\lim_\alpha\lim_\beta\lVert\eta_\alpha+\xi_\beta\rVert$ and $\lim_\beta\lim_\alpha\lVert\eta_\alpha+\xi_\beta\rVert$ exist. By passing to suitable subnets, it is enough to assume that $\lim_\alpha\left\lVert \eta_\alpha\right\rVert$ and 
$\lim_\beta\left\lVert\xi_\beta\right\rVert$ exist.
Since it holds for any $x\in S_X$ and non-zero $\eta_\alpha$ that
$$
\|\eta_\alpha\|+\|\eta_\alpha\| \bar{\delta}_X\left(\frac{\lim_\beta\|\xi_\beta\|}{\|\eta_\alpha\|}\right)
\leq \lim_\beta \left\|\|\eta_\alpha\| x+\xi_\beta\right\| \leq 
\|\eta_\alpha\|+\|\eta_\alpha\| \bar{\rho}_X\left(\frac{\lim_\beta \|\xi_\beta\|}{\|\eta_\alpha\|}\right),
$$
we get that
$$\lim_\beta \left\|\|\eta_\alpha\| x+\xi_\beta\right\| = N^A_X\left(\|\eta_\alpha\|, \lim_\beta \|\xi_\beta\|\right) \text{~for every~} x\in S_X.$$
 Similarly, we have that
$$\lim_\alpha \left\|\|\xi_\beta\|x+\eta_\alpha\right\| = N^A_X\left(\|\xi_\beta\|, \lim_\alpha \|\eta_\alpha\|\right) \text{~for every~} x\in S_X.$$ 
Hence, we deduce that 

\begin{align*}
\lim_\alpha\lim_\beta\left\lVert\eta_\alpha+\xi_\beta\right\rVert
&=\lim_\alpha\lim_\beta\left\lVert\left\lVert\eta_\alpha\right\rVert x+\xi_\beta\right\rVert
=N^{A}_X\left(\lim_\alpha\left\lVert \eta_\alpha\right\rVert,\lim_\beta\left\lVert\xi_\beta\right\rVert\right)\\
&=N^{A}_X\left(\lim_\beta\left\lVert\xi_\beta\right\rVert,\lim_\alpha\left\lVert \eta_\alpha\right\rVert\right)
=\lim_\beta\lim_\alpha \left\|\|\xi_\beta\|x+\eta_\alpha\right\|\\
&=\lim_\beta\lim_\alpha\left\lVert\eta_\alpha+\xi_\beta\right\rVert.
\end{align*}
\end{proof}
\end{theorem}

According to the aforementioned results of N. J. Kalton in \cite{K1} and E. Oja in \cite{Oja1}, whenever a Banach space $X$ satisfies that $\mathcal{K}(X)$ is an M-ideal in $\mathcal{L}(X)$, $X$ has property $(M_p)$ for some $1<p<\infty$ if and only if there exists $1 < p <\infty$ such that $X$ has property $(m_p)$. Especially, whenever $X$ is separable, these conditions are equivalent to the stability of $X$. We observe that the same characterization holds with the net-stability for non-separable spaces as follows.

\begin{cor}\label{Mpeq}  For a Banach space $X$ such that $\mathcal{K}(X)$ is an M-ideal in $\mathcal{L}(X)$, the following are equivalent. 
\begin{enumerate}[(i)]
\item $X$ has property $(M_p)$ for some $1 < p < \infty$.
\item $X$ is net-stable.
\item $X$ is reflexive and weak-net-stable.
\item $X$ is finite dimensional, or $X$ is reflexive and satisfies $\bar{\delta}_X=\bar{\rho}_X$.
\item There exists $1 < p <\infty$ such that $X$ has property $(m_p)$.
\end{enumerate}
\end{cor}
\begin{proof} Since every finite dimensional space satisfies all the conditions clearly, it is enough to assume that $X$ is infinite dimensional.

 E. Oja and D. Werner proved that if a separable Banach space $X$ have property $(M_p)$, then $X$ is stable (see \cite[Proposition 2.5]{OW}). Since the proof of the implication $(i)~\Rightarrow (ii)$ is just a simple modification of the one of this, we omit it. The only difference is that we need to use nets instead of sequences. The implications $(v)~\Rightarrow (i)$, $(iv)~\Rightarrow (v)$ and $(iii)~\Rightarrow (iv)$ are shown in \cite[Corollary 6]{Oja1}, corollary \ref{asyp} and  theorem \ref{Msym} respectively.

 In order to see $(ii)~\Rightarrow (iii)$, we only need to prove that $X$ is reflexive. Assume that $X$ is net-stable, and note that every separable subspace of $X$ is stable. Since $X$ has the MCAP (see \cite[Theorem 2]{Oja1}), for any separable subspace $Y$ of $X$ there exists a separable subspace $Z$ of $X$ containing $Y$ having the MCAP (see \cite[P606]{S} or \cite[Theorem 4.1]{OZ}). We also see that $\mathcal{K}(Z)$ is an M-ideal in $\mathcal{L}(Z)$ from \cite[Theorem 2]{Oja1} again. Since $Z$ is stable, it has property $M_p$ for some $1<p<\infty$ (see \cite{K1}). Hence, $Z$ is reflexive by \cite[Proposition 5.2, Ch. VI]{HWW}. This shows that  every separable subspace of $X$ is reflexive, and so is $X$.
 \end{proof}

\subsection*{Funding}
The first and second author were supported by the National Research Foundation of Korea(NRF) grant funded by the Korea government (MSIT) [NRF-2020R1C1C1A01012267].


\begin{thebibliography}{99}

\bibitem{AE} E. M. Alfsen and E. G. Effros, \emph{Structure in real Banach spaces} Ann. of Math. (2) \textbf{96} (1972), 98-173.
\bibitem{AGPT}R. M. Aron, D. Garc\'{\i}a, D. Pellegrino, and E. V. Teixeira, \emph{Reflexivity and nonweakly null maximizing sequences}, Proc. Amer. Math. Soc. \textbf{148} (2020), no.~2, 741-50.

\bibitem{CJ} C. Cho and W. B. Johnson, \emph{A characterization of subspaces  $X$  of  $\ell_p$  for which  $K(X)$  is an  M-ideal in  $L(X)$}, Proc. Amer. Math. Soc. \textbf{93} (1985), no. 3, 466–470.

\bibitem{DJM} S. Dantas, M. Jung and G. Martínez-Cervantes\emph{Some remarks on the weak maximizing property.} J. Math. Anal. Appl., (2021) \textbf{504} (2), 125433.
\bibitem{Dul} D. van Dulst, \emph{Equivalent norms and the fixed point property for nonexpansive mappings.} J. London Math. Soc. (2) \textbf{25} (1982), no.1, 139-144.

\bibitem{VL} V. P. Fonf and L. Vesel\'{y}, \emph{Infinite-dimensional polyhedrality}, Canad. J. Math. \textbf{56} (2004), no.~3, 472-494.
\bibitem{GP} L. C. Garc\'{\i}a-Lirola and C. Petitjean, \emph{On the weak maximizing properties}, Banach J. Math. Anal. \textbf{15} (2021), no. 3, Paper No. 55.

\bibitem{Ha} J. D. Hardtke, \emph{WORTH property, Garcia-Falset coefficient and Opial property of infinite sums.}, Comment. math. \textbf{55} (2015), no.~1, 23-44.
\bibitem{HL} P. Harmand and \AA. Lima, \emph{Banach spaces which are M-ideals in their biduals} Trans. Amer. Math. Soc. \textbf{283} (1984), no.1, 253-264.
\bibitem{HWW} P. Harmand, D. Werner and W. Werner,  \emph{M -ideals in Banach spaces and Banach algebras}, Lecture Notes in Math., 1547 Springer-Verlag, Berlin, 1993, viii+387 pp. ISBN: 3-540-56814-X

\bibitem{H1} J. Hennefeld, \emph{A decomposition of $B(X)^*$ and unique Hahn-Banach extensions}, Pacific J. Math. \textbf{46} (1973), 197-199.
\bibitem{H2} J. Hennefeld, \emph{M-ideals, HB-subspaces, and compact operators}, Indiana Univ. Math. J. \textbf{28} (1979), 927-934.


\bibitem{JW} K. John and D. Werner, \emph{$M$-ideals of compact operators into $l_p$}, Czechoslovak Math. J. \textbf{50(125)} (2000), no.~1, 51-57.
\bibitem{JLDS} W. B. Johnson, J. Lindenstrauss, D. Preiss, and G. Schechtman, \emph{Almost {F}r\'{e}chet differentiability of {L}ipschitz mappings between
  infinite-dimensional {B}anach spaces}, Proc. London Math. Soc. (3) \textbf{84} (2002), no.~3, 711-746.


\bibitem{GMA} M. Jung, G. Mart\'inez-Cervantes and A. Rueda Zoca, \emph{Rank-one perturbations and norm-attaining operators}, Math. Z, 
To appear.
\bibitem{K1} N. J. Kalton, \emph{M -ideals of compact operators}, Illinois J. Math. \textbf{37} (1993), no. 1, 147-169.
\bibitem{K2} N. J. Kalton, \emph{Uniform homeomorphisms of Banach spaces and asymptotic structure}, Trans. Amer. Math. Soc., \textbf{365} (2013), pp. 1051–1079.

\bibitem{KW} N. J. Kalton and D. Werner, \emph{Property {$(M)$}, {$M$}-ideals, and almost isometric structure of Banach spaces}, J. Reine Angew. Math. \textbf{461} (1995), 137-178.
\bibitem{KM} J. L. Krivine and B. Maurey, \emph{Espaces de Banach stables}, Israel J. Math. \textbf{39} (1981) 273-295.

\bibitem{L1} \AA. Lima, \emph{M-ideals of compact operators in classical Banach spaces}, Math. Scand. \textbf{44} (1979), 207-217.
\bibitem{L2} \AA. Lima, \emph{On M-ideals and best approximation}, Indiana Univ. Math. J. \textbf{31} (1982), 27-36.
\bibitem{LORW} \AA. Lima, E. Oja, T. S. S. R. K. Rao, and D. Werner, \emph{Geometry of operator spaces}, Michigan Math. J. 41 (1994), 473-490. 

\bibitem{LT} J. Lindenstrauss and L. Tzafriri, \emph{Classical {B}anach spaces. I}, Springer-Verlag, Berlin-New York, 1977.
\bibitem{GA} G. L\'{o}pez-P\'{e}rez and A. Rueda Zoca, \emph{Diameter two properties and polyhedrality}, Rev. R. Acad. Cienc. Exactas F\'{i}s. Nat. Ser. A Mat. RACSAM \textbf{113} (2019), no. 1, 131-135.

\bibitem{Oja1}E. Oja, \emph{{$M$}-ideals of compact operators are separably determined}, Proc. Amer. Math. Soc. \textbf{126} (1998), no.~9, 2747-2753.
\bibitem{OW} E. Oja and D. Werner, \emph{Remarks on M-ideals of compact operators on $X\oplus_p X$} Math. Nachr. \textbf{152} (1991), 101-111.
\bibitem{OZ} E. Oja and I. Zolk, \emph{The asymptotically commuting bounded approximation property of Banach spaces}, J. Funct. Anal. \textbf{266} (2014), no. 2, 1068-1087.
\bibitem{Op} Z. Opial, \emph{Weak convergence of the sequence of successive approximations for nonexpansive mappings}, Bull. Amer. Math. Soc. \textbf{73} (1967), 591-597.


\bibitem{Sim} B. Sims, \emph{A class of spaces with weak normal structure}, Bull. Austral. Math. Soc. \textbf{49} (1994), no.~3, 523-528. 

\bibitem{S} I. Singer, \emph{Bases in Banach spaces II}, Editura Acad. R. S. Rom\^{a}nia, BucharestSpringer-Verlag, Berlin-New York, 1981, viii+880


\bibitem{SW} R. R. Smith and J. D. Ward, \emph{M-ideal structure in Banach algebras}, J. Funct. Anal. \textbf{27} (1978), 337-349.


\bibitem{WWerner} W. Werner,  \emph{Smooth points in some spaces of bounded operators.}, Integral Equations Operator Theory \textbf{15} (1992), no. 3, 496-502.

\bibitem{Xu} H. Xu, G. Marino, and P.Pietramala, \emph{On property {$(M)$} and its generalizations}, J. Math. Anal. Appl. \textbf{261} (2001), no.~1, 271-281.



\end{thebibliography}
\end{document}